\documentclass[10pt,reqno]{amsart}
   \usepackage[centertags]{amsmath}
   \usepackage{amsfonts}
   \usepackage{amsmath}
   \usepackage{amssymb}
   \usepackage{amsthm}
   \usepackage{newlfont}
   \usepackage[latin1]{inputenc}

\usepackage{multirow, bigdelim}

\usepackage{epsfig}
\usepackage{pst-all}
\usepackage{pstricks}
\usepackage{pgf}
\usepackage{mathrsfs}
\usepackage{hyperref}
\usepackage{bm}

\usepackage{graphicx}

\usepackage{bm}

\usepackage{mathtools}
\usepackage{ifmtarg}

\makeatletter

    \newcommand\contFrac{\@ifstar{\@contFracStar}{\@contFracNoStar}}

    \def\singleContFrac#1#2{%
        \begin{array}{@{}c@{}}%
            \multicolumn{1}{c|}{#1}%
            \\%
            \hline%
            \multicolumn{1}{|c}{#2}%
        \end{array}%
    }

    \def\@contFracNoStar#1{%
        \mathchoice{
            \@contFracNoStarDisplay@#1//\@nil%
        }{
            \@contFracNoStarInline@#1//\@nil%
        }{
            \@contFracNoStarInline@#1//\@nil%
        }{
            \@contFracNoStarInline@#1//\@nil%
        }%
    }

    \def\@contFracNoStarDisplay@#1//#2\@nil{%
        \@ifmtarg{#2}{%
            #1%
        }{%
            #1+\cfrac{1}{\@contFracNoStarDisplay@#2\@nil}%
        }%
    }

        \def\@contFracNoStarInline@#1//#2\@nil{%
            \@ifmtarg{#2}{%
                #1%
            }{%
                #1 \@@contFracNoStarInline@@#2\@nil%
            }%
        }
        \def\@@contFracNoStarInline@@#1//#2\@nil{%
            \@ifmtarg{#2}{%
                + \singleContFrac{1}{#1}%
            }{%
                + \singleContFrac{1}{#1} \@@contFracNoStarInline@@#2\@nil%
            }%
        }

    \def\@contFracStar#1{%
        \mathchoice{
            \@contFracStarDisplay@#1////\@nil%
        }{
            \@contFracStarInline@#1//\@nil%
        }{
            \@contFracStarInline@#1//\@nil%
        }{
            \@contFracStarInline@#1//\@nil%
        }%
    }

    \def\@contFracStarDisplay@#1//#2//#3\@nil{%
        \@ifmtarg{#2}{%
            #1%
        }{%
            #1 + \cfrac{#2}{\@contFracStarDisplay@#3\@nil}%
        }%
    }

        \def\@contFracStarInline@#1//#2\@nil{%
            \@ifmtarg{#2}{%
                #1%
            }{%
                #1 \@@contFracStarInline@@#2\@nil%
            }%
        }
        \def\@@contFracStarInline@@#1//#2//#3\@nil{%
            \@ifmtarg{#3}{%
                - \singleContFrac{#1}{#2}%
            }{%
                - \singleContFrac{#1}{#2} \@@contFracStarInline@@#3\@nil%
            }%
        }
\makeatother

\def\cFrac#1#2{%
\begin{array}{@{}c@{}}\multicolumn{1}{c|}{#1}\\%
\hline\multicolumn{1}{|c}{#2}\end{array}}

\allowdisplaybreaks[3]

\theoremstyle{plain}
\newtheorem{theorem}{Theorem}[section]
\newtheorem{lemma}[theorem]{Lemma}
\newtheorem{proposition}[theorem]{Proposition}

\theoremstyle{definition}

\theoremstyle{remark}
\newtheorem{remark}[theorem]{Remark}
\newtheorem*{remark*}{Remark}

\numberwithin{equation}{section}

\newcommand\D{\displaystyle}

\newcommand\ZZ{{\mathbb Z}}

\title[Stochastic Darboux for random walks on $\mathbb{Z}$]{The spectral matrices associated with the stochastic Darboux transformations of random walks on the integers}

\author{Manuel D. de la Iglesia}
\address{Manuel D. de la Iglesia\\
Instituto de Matem\'aticas, Universidad Nacional Aut\'onoma de M\'exico, Circuito Exterior, C.U., 04510, Ciudad de M\'exico, M\'exico.}
\email{mdi29@im.unam.mx}

\author{Claudia Juarez}
\address{Claudia Juarez\\
Instituto de Matem\'aticas, Universidad Nacional Aut\'onoma de M\'exico, Circuito Exterior, C.U., 04510, Ciudad de M\'exico, M\'exico.}
\email{ClaudiaJrz@ciencias.unam.mx}
\date{\today}
\thanks{
}

\thanks{This work was partially supported by PAPIIT-DGAPA-UNAM grant IN104219 (M\'exico), UC MEXUS-CONACYT grant CN-16-84 and CONACYT grant A1-S-16202 (M\'exico).}

\date{\today}



\begin{document}

\maketitle

\begin{abstract}

We consider UL and LU stochastic factorizations of the transition probability matrix of a random walk on the integers, which is a doubly infinite tridiagonal stochastic Jacobi matrix. We give conditions on the free parameter of both factorizations in terms of certain continued fractions such that this stochastic factorization is always possible. By inverting the order of the factors (also known as a Darboux transformation) we get new families of random walks on the integers. We identify the spectral matrices associated with these Darboux transformations (in both cases) which are basically conjugations by a matrix polynomial of degree one of a Geronimus transformation of the original spectral matrix. Finally, we apply our results to the random walk with constant transition probabilities with or without an attractive or repulsive force.

\end{abstract}

\section{Introduction}

The main goal of this paper is to describe the spectral matrices associated with the discrete Darboux transformations of the one-step transition probability matrix $P$ of a random walk on the integers $\ZZ$. This transition probability matrix is a doubly infinite tridiagonal stochastic matrix, also known as a Jacobi matrix or Jacobi operator (see \eqref{QZ} below) acting on the space $\ell_{\pi}^2(\ZZ)$ for certain sequence $\pi=(\pi_n)_{n\in\ZZ}$. The case of random walks on the nonnegative integers $\ZZ_{\geq0}$ has been recently considered by F.A. Gr\"unbaum and one of the authors of this paper in \cite{GdI3}. The main motivation for this UL and LU stochastic factorization is to divide the probabilistic model associated with the random walk into two different and simpler experiments, and combine them together to obtain a simpler description of the original probabilistic model (see applications to urn models in \cite{GdI3}).

We start by analyzing the conditions under we can perform a stochastic UL factorization of the form $P=P_UP_L$ or a stochastic LU factorization of the form $P=\widetilde P_L\widetilde P_U$, where all factors (bidiagonal matrices) are also stochastic matrices. In the case of the UL factorization, as the situation of random walks on $\ZZ_{\geq0}$, we still have one free parameter. But one important difference now is that for the LU factorization we also have one \emph{free parameter}, something that did not happen for random walks on $\ZZ_{\geq0}$, where the factorization was unique. In both cases, the factorization, if it can be achieved in terms of stochastic factors, will represent a family of factorizations of the original transition probability matrix $P$. In \cite{GdI3} it is shown that this free parameter has to be bounded from above by certain continued fraction if we want to guarantee that the factors are still stochastic matrices. In our case, since we are dealing with doubly infinite stochastic matrices, this free parameter (in both cases) has to be not only bounded from above but also bounded from below by another continued fraction which is built from the negative states of the original random walk. This will be the content of Section \ref{sec2}. UL and LU factorizations of stochastic matrices have been considered earlier in the literature (see for instance \cite{Gr1, Gr2, Hey, Vig}) but these factorizations are different from the one we try to consider here where all matrices involved are stochastic (see \cite{GdI3} for an extended discussion about this matter).

Once we have a stochastic UL or LU factorization of a tridiagonal stochastic matrix we can make use of the so-called \emph{discrete Darboux transformation}, consisting of inverting the order of the factors. The new matrices $\widetilde P=P_LP_U$ and $\widehat P=\widetilde P_U\widetilde P_L$ will also be doubly infinite tridiagonal and stochastic matrices. Since both factorizations come with one free parameter, we will have a family of new random walks different in general from the original one. These discrete Darboux transformations have been studied before in the context of the theory of orthogonal polynomials, in particular in the description of some families of Krall polynomials (see \cite{GH, GHH, SZ, Yo, Ze}). It has played an important role in the study of integrable systems (see \cite{MS}). An important issue that has not been considered before, as far as the authors know, is how to relate the spectral matrix $\Psi(x)$ associated with $P$ with the spectral matrices $\widetilde\Psi(x)$ or $\widehat\Psi(x)$ associated with $\widetilde P$ or $\widehat P$, respectively. By spectral matrix we mean that the spectral analysis of $P$ comes now with \emph{three measures}  $\psi_{\alpha,\beta}, \alpha,\beta=1,2,$ (two positive and one signed measure, and $\psi_{12}=\psi_{21}$ due to the symmetry) and they can be written in a $2\times 2$ spectral matrix of the form
\begin{align*}
\Psi(x)&=\begin{pmatrix} \psi_{11}(x) & \psi_{12}(x)\\ \psi_{12}(x) & \psi_{22}(x) \end{pmatrix},
\end{align*}
which turns out to be a proper weight matrix in the context of the theory of matrix-valued orthogonal polynomials. In fact, a random walk in $\ZZ$ can be viewed as a special type of discrete-time \emph{quasi-birth-and-death process} with state space $\ZZ_{\geq0}\times\{1,2\}$. These processes can be defined in general in state spaces of the form $\ZZ_{\geq0}\times\{1,\ldots, N\}$ for $N\geq1$ a positive integer (see \cite{LaR, Neu} for general references). The spectral analysis of these processes has been considered for instance in \cite{DRSZ, G2, G1, GdI2, dIR} to mention a few.

For random walks on $\mathbb{Z}_{\geq0}$ it is very well-known that the spectral measure associated with the Darboux transformation of a UL factorization is given by a so-called \emph{Geronimus transformation} of the original spectral measure, while for the LU factorization is given by a \emph{Christoffel transformation}. For random walks on $\mathbb{Z}$ we will show that the spectral matrices associated with these Darboux transformations (in both cases) are conjugations of the form
\begin{equation*}
\widetilde \Psi (x)=\bm S_0(x) \Psi_S (x) \bm S_0^*(x),\quad \widehat \Psi (x)=\bm T_0(x) \Psi_T (x) \bm T_0^*(x),
\end{equation*}
where $\Psi_S (x), \Psi_T (x)$ are Geronimus transformations of the original spectral matrix $\Psi$ and $\bm S_0(x), \bm T_0(x)$ are certain matrix polynomials of degree one (see Theorems \ref{thmorto} and \ref{thmorto2} in Section \ref{sec3}). In \cite{GdI4} a first attempt has been done to study stochastic Darboux transformations of block tridiagonal stochastic matrices, which are the transition probability matrices of discrete-time quasi-birth-and-death processes. In that paper the authors only consider one (Jacobi type) example previously introduced in \cite{GPT1} (see also \cite{GPT2, GPT3}). The UL factorization depends now on one free matrix-valued parameter and it is not clear how to transform the corresponding spectral weight matrices associated with the discrete Darboux transformation in general. In the case of random walks on $\ZZ$ we will only have one free (real) parameter. Our results can give some insights about how to compute the spectral matrix of the Darboux transformations for tridiagonal stochastic block matrices in general. For a different application of the Darboux transformation in the context of the noncommutative bispectral problem see \cite{GHY, G3, Zu} and references therein.

Once we have the spectral matrix it is easy to analyze the corresponding random walk in terms of the two independent families of polynomials which arise as a solution of the eigenvalue equation. For the case of random walks on $\mathbb{Z}_{\geq0}$ this was first done in a series of papers by S. Karlin and J. McGregor (inspired by work by W. Feller and H.P. McKean) in the 1950s (see \cite{KMc2, KMc3, KMc6}) where they studied first continuous-time birth-and-death processes and then the case of discrete-time random walks. Apart from an explicit expression of the $n$-step transition probabilities and the invariant measure, it is possible to study some other probabilistic properties using spectral methods such as recurrence, absorbing times, first return times or limit theorems. In the last section of \cite{KMc6} one can find the first attempt to perform the spectral analysis of a random walk on $\ZZ$ using orthogonal polynomials. We will recall this approach at the beginning of Section \ref{sec3}. After that, apart from \cite{Ber}, there are not so many references concerning the spectral analysis of doubly infinite Jacobi operators acting on $\ell_{\pi}^2(\ZZ)$. In \cite{Pru}, W.E. Pruitt studied the case of birth-and-death processes on $\ZZ$ (also known as bilateral birth-and-death processes). An example of this approach can be found in the last section of \cite{ILMV}. A more theoretical work about the spectral theory of Jacobi operators acting on $\ell_{\pi}^2(\ZZ)$ was given by D.E. Masson and J. Repka in \cite{MR} and revisited recently in \cite{DIW}. 

Finally, we will apply our results to two examples in Section \ref{sec4}. The first one is the random walk on $\ZZ$ with constant transition probabilities, while the second one is the random walk on $\ZZ$ with constant transition probabilities but allowing an attractive or repulsive force to or from the origin. In both cases we study the conditions under we get a stochastic UL and LU factorization, give the corresponding spectral matrices and the spectral matrices associated with both discrete Darboux transformations. As a final remark we will show that it is possible to choose certain values of the free parameters such that the Darboux transformation (both from the UL or the LU factorization) is \emph{invariant}, i.e. we get the same random walk after we perform the Darboux transformation. This phenomenon it is not possible for Darboux transformations of random walks on $\ZZ_{\geq0}$.

\section{Stochastic UL and LU factorization on the integers}\label{sec2}

Let $\{X_t : t=0,1,\ldots\}$ be an irreducible random walk on the integers $\mathbb{Z}$ with transition probability matrix $P$  given by
\begin{equation}\label{QZ}
P=
\left(
\begin{array}{ccc|cccc}
\ddots&\ddots&\ddots&&&\\
&c_{-1}&b_{-1}&a_{-1}&&&\\
\hline
&&c_0&b_0&a_0&&\\
&&&c_1&b_1&a_1&\\
&&&&\ddots&\ddots&\ddots
\end{array}
\right).
\end{equation}
The matrix $P$ is stochastic, i.e. all entries are nonnegative and
$$
c_n+b_n+a_n=1,\quad n\in\ZZ.
$$
Since the random walk is irreducible then we have that $0<a_n,c_n<1, n\in\ZZ$. A diagram of the transitions between the states is given by
\begin{center}
\includegraphics[scale=.75]{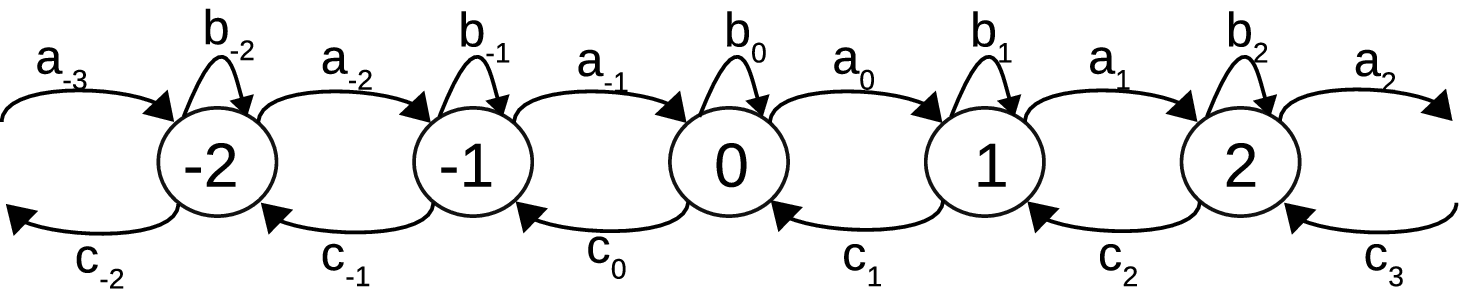}
\end{center}
Let us perform a UL factorization of $P$ in the following way
\begin{equation}\label{QZUL}
P=\left(\begin{array}{cc|cccc}
\ddots&\ddots&&\\
0&y_{-1}&x_{-1}&&\\
\hline
&0&y_0&x_0&\\
&&0&y_1&x_1\\
&&&&\ddots&\ddots
\end{array}
\right)\left(\begin{array}{ccc|ccc}
\ddots&\ddots&&&\\
&r_{-1}&s_{-1}&0&\\
\hline
&&r_0&s_0&0\\
&&&r_1&s_1&0\\
&&&&\ddots&\ddots
\end{array}
\right)=P_UP_L,
\end{equation}
where $P_U$ and $P_L$ are also stochastic matrices. This means that all entries of $P_U$ and $P_L$ are nonnegative and 
\begin{equation}\label{stpul}
x_n+y_n=1,\quad s_n+r_n=1,\quad n\in\ZZ.
\end{equation}
A direct computation shows that
\begin{align}
\nonumber a_n&=x_ns_{n+1},\\
\label{ULd}b_n&=x_nr_{n+1}+y_ns_n,\quad n\in\ZZ,\\
\nonumber c_n&=y_nr_n.
\end{align}
By the irreducibility conditions we immediately have that $0<x_n,y_n,s_n,r_n<1, n\in\ZZ$. The Markov chain associated with $P_U$ is a pure birth random walk on $\ZZ$ with diagram
\begin{center}
\includegraphics[scale=.75]{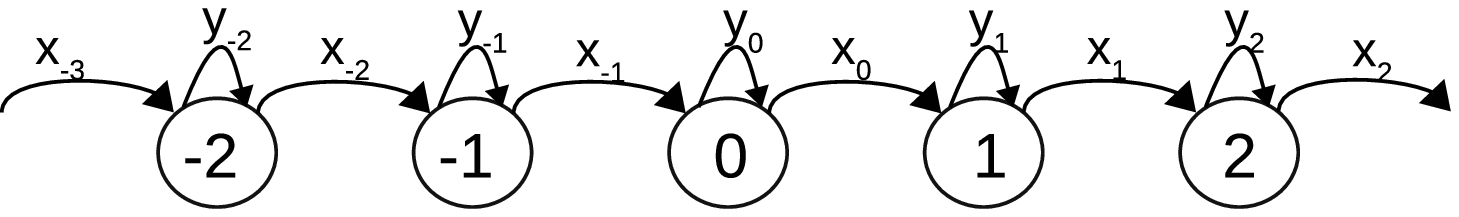}
\end{center}
while $P_L$ is a pure death random walk on $\ZZ$ with diagram
\begin{center}
\includegraphics[scale=.75]{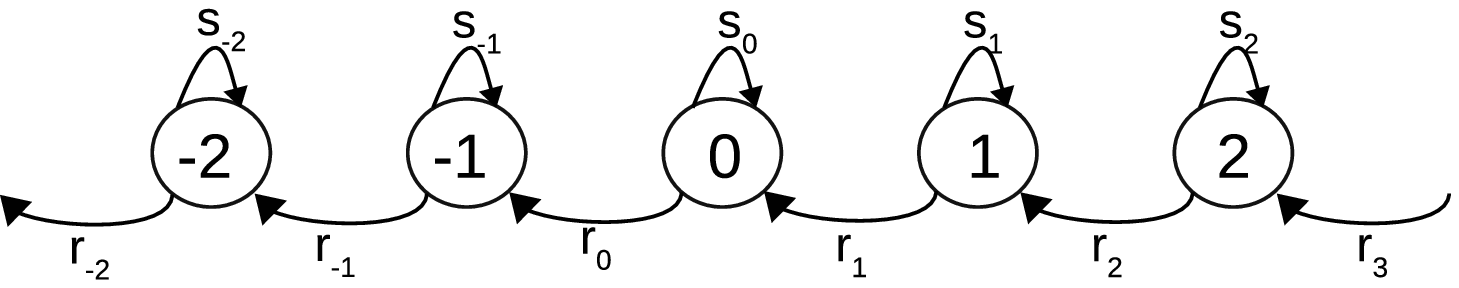}
\end{center}
As in the case of random walks on $\ZZ_{\geq0}$ (see \cite{GdI3}) we can compute all entries of $P_U$ and $P_L$ in terms of only \emph{one free parameter}, namely $y_0$. Indeed, for nonnegative values of the indices and $y_0$ fixed we can compute $x_0, s_1, r_1, y_1, x_1, s_2, r_2, y_2,\ldots$ recursively using \eqref{stpul} and \eqref{ULd}. Similarly, for negative values of the indices and $y_0$ fixed we can compute $r_0, s_0, x_{-1}, y_{-1}, r_{-1}, s_{-1}, x_{-2}, y_{-2},\ldots$ recursively using again \eqref{stpul} and  \eqref{ULd}. 

Following the same steps as in Lemma 2.1 of \cite{GdI3} we have that for a decomposition like in \eqref{QZUL}, i.e. $P=P_UP_L$, the matrix $P_U$ is stochastic if and only if the matrix $P_L$ is stochastic. We also have, following \eqref{ULd}, the following relations
\begin{equation*}\label{relys}
y_n=\frac{c_n}{1-s_n},\quad s_{n+1}=\frac{a_{n}}{1-y_n},\quad n\in\ZZ.
\end{equation*}
Although we can compute all coefficients $x_n, y_n, s_n, r_n$ in terms of one free parameter $y_0$, we can not infer anything about the positivity of these coefficients. This will be the goal of the next theorem. As it was done in Theorem 2.1 of \cite{GdI3} for random walks on $\ZZ_{\geq0}$ the matrices $P_U$ and $P_L$ are stochastic if and only if $0\leq y_0\leq H$, where $H$ is the continued fraction given below by \eqref{cfUL}. The difference now for random walks on $\ZZ$ is that the free parameter $y_0$ will be also bounded below by another number $H'$ defined by a continued fraction generated by the probabilities of the negative states of the random walk. We recommend the reference \cite{Wa} for the reader unfamiliar with continued fractions (as well as \cite{GdI3} for the case of random walks on $\mathbb{Z}_{\geq0}$).

Let $H$ and $H'$ be the continued fractions generated by alternatively choosing $a_n$ and $c_n$ in different directions, i.e.
\begin{equation}\label{cfUL}
H=1-\cfrac{a_0}{1-\cfrac{c_1}{1-\cfrac{a_1}{1-\cfrac{c_2}{1-\cdots}}}},\quad H'=\cfrac{c_0}{1-\cfrac{a_{-1}}{1-\cfrac{c_{-1}}{1-\cfrac{a_{-2}}{1-\cdots}}}}.
\end{equation}
In a different notation
\begin{equation*}
H=1-\cFrac{a_0}{1}-\cFrac{c_1}{1}-\cFrac{a_1}{1}-\cFrac{c_2}{\cdots},\quad H'=\cFrac{c_0}{1}-\cFrac{a_{-1}}{1}-\cFrac{c_{-1}}{1}-\cFrac{a_{-2}}{\cdots}.
\end{equation*}
For each continued fraction, consider the corresponding sequence of \emph{convergents} $(h_{n})_{n\geq0}$ and $(h_{-n}')_{n\geq0}$, given by
\begin{equation}\label{hhh}
h_{n}=\frac{A_n}{B_n},\quad h_{-n}'=\frac{A_{-n}'}{B_{-n}'}.
\end{equation}
Recall that the convergents of a continued fraction $H (H')$ are the sequence of truncated continued fractions of $H (H')$ and these are always rational numbers. In \cite{GdI3} (see also \cite{Wa}) it is proved that the numbers $A_n,B_n$ can be recursively obtained using the following formulas
\begin{align*}
A_{2n}&=A_{2n-1}-c_nA_{2n-2},\quad n\geq1,\quad A_{2n+1}=A_{2n}-a_nA_{2n-1},\quad n\geq0,\quad A_{-1}=1,\quad A_0=1,\\
B_{2n}&=B_{2n-1}-c_nB_{2n-2},\quad n\geq1,\quad B_{2n+1}=B_{2n}-a_nB_{2n-1},\quad n\geq0,\quad B_{-1}=0,\quad B_0=1.
\end{align*}
In the same way the numbers $A_{-n}',B_{-n}'$ can be recursively obtained using
\begin{align*}
A_{-2n}'&=A_{-2n+1}'-a_{-n}A_{-2n+2}',\quad n\geq1,\quad A_{-2n-1}'=A_{-2n}'-c_{-n}A_{-2n+1}',\quad n\geq0,\quad A_{1}'=-1,\quad A_0'=0,\\
B_{-2n}'&=B_{-2n+1}'-a_{-n}B_{-2n+2}',\quad n\geq1,\quad B_{-2n-1}'=B_{-2n}'-c_{-n}B_{-2n+1}',\quad n\geq0,\quad B_{1}'=0,\quad B_0'=1.
\end{align*}
Using these relations it is not hard to prove that
\begin{equation}\label{rels2}
\begin{split}
A_{-2n}'B_{-2n-1}'-A_{-2n-1}'B_{-2n}'&=-c_0a_{-1}c_{-1}\cdots a_{-n}c_{-n},\quad n\geq0,\\
A_{-2n-1}'B_{-2n-2}'-A_{-2n-2}'B_{-2n-1}'&=-c_0a_{-1}c_{-1}\cdots c_{-n}a_{-n-1},\quad n\geq0.
\end{split}
\end{equation}
Then we have the following:
\begin{theorem}\label{thmfracUL}
Let $H$ and $H'$ be the continued fractions given by \eqref{cfUL} and the corresponding convergents $h_n$ and $h_{-n}$ defined by \eqref{hhh}. Assume that 
\begin{equation}\label{AnBn}
0<A_n<B_n,\quad\mbox{and}\quad0<A_{-n}'<B_{-n}',\quad n\geq1.
\end{equation}
Then both $H$ and $H'$ are convergent. Moreover, let $P=P_UP_L$ as in \eqref{QZUL}. Assume that $H'\leq H$. Then, both $P_U$ and $P_L$ are stochastic matrices if and only if we choose $y_0$ in the following range
\begin{equation}\label{yy0r}
H'\leq y_0\leq H.
\end{equation}
\end{theorem}
\begin{proof}
The convergence of $H$ and the upper bound for $y_0$ is proved in Theorem 2.1 of \cite{GdI3}. From \eqref{rels2} and using the assumptions \eqref{AnBn} we have that
\begin{align*}
h_{-2n-2}'-h_{-2n-1}'&=\frac{A_{-2n-2}'}{B_{-2n-2}'}-\frac{A_{-2n-1}'}{B_{-2n-1}'}=\frac{c_0a_{-1}c_{-1}\cdots c_{-n}a_{-n-1}}{B_{-2n-1}'B_{-2n-2}'}>0,\quad n\geq0,\\
h_{-2n-1}'-h_{-2n}'&=\frac{A_{-2n-1}'}{B_{-2n-1}'}-\frac{A_{-2n}'}{B_{-2n}'}=\frac{c_0a_{-1}c_{-1}\cdots a_{-n}c_{-n}}{B_{-2n}'B_{-2n-1}'}>0,\quad n\geq0.
\end{align*}
Therefore we have
$$
0=h_0'<h_{-1}'<h_{-2}'<\cdots<h_{-2n}'<h_{-2n-1}'<h_{-2n-2}'<\cdots<1,
$$
that is, $(h_{-n}')_{n\geq0}$ is a bounded strictly increasing sequence, so it is convergent to $H'$.

For the lower bound for $y_0$, assume first that both $P_U$ and $P_L$ are stochastic matrices, so that $0<x_n,y_n,s_n,r_n<1,n\in\ZZ$. Then we have, using  \eqref{ULd}, that
\begin{align*}
s_0=&1-\frac{c_0}{y_0}>0 \Leftrightarrow \frac{c_0}{y_0}<1 \Leftrightarrow  c_0<y_0 \Leftrightarrow h_{-1}'<y_0,
\end{align*}
and
\begin{align*}
y_{-1}=&1-\frac{a_{-1}}{s_0}>0  \Leftrightarrow \frac{a_{-1}}{s_0}<1 \Leftrightarrow  a_{-1} <s_0
 \Leftrightarrow a_{-1}< 1-\frac{c_0}{y_0}\\& \quad \Leftrightarrow 1-a_{-1}>\frac{c_0}{y_0}  \Leftrightarrow  \frac{1}{1-a_{-1}}<\frac{y_0}{c_0} \Leftrightarrow \frac{c_0}{1-a_{-1}}<y_0 \Leftrightarrow h_{-2}'<y_0.
\end{align*}
Following the same argument we have
\begin{equation*}
\begin{split}
y_{-n}=1-\frac{a_{-n}}{s_{-n+1}}>0& \Leftrightarrow  a_{-n}<s_{-n+1} \Leftrightarrow 1-a_{-n}>\frac{c_{-n+1}}{y_{-n+1}} \Leftrightarrow \frac{c_{-n+1}}{1-a_{-n}}<y_{-n+1} \\
&\Leftrightarrow \frac{c_{-n+1}}{1-a_{-n}}<1-\frac{a_{-n+1}}{s_{-n+2}} \Leftrightarrow \frac{a_{-n+1}}{1-\D\frac{c_{-n+1}}{1-a_{-n}}}<s_{-n+2} \\
&\cdots \Leftrightarrow \cFrac{a_{-n+n-1}}{1}-\cFrac{c_{-1}}{1}-\cFrac{a_{-2}}{1}\cdots -\cFrac{a_{-n}}{1} <s_{-n+n}\\
&\Leftrightarrow \cFrac{a_{-1}}{1}-\cFrac{c_{-1}}{1}-\cFrac{a_{-2}}{1}\cdots -\cFrac{a_{-n}}{1} <1-\frac{c_0}{y_0}\\
&\Leftrightarrow \cfrac{c_0}{1}- \cFrac{a_{-1}}{1}-\cFrac{c_{-1}}{1}-\cFrac{a_{-2}}{1}\cdots -\cFrac{a_{-n}}{1} <y_0 \\
&\Leftrightarrow h_{-2n}'<y_0.
\end{split}
\end{equation*}
Therefore we always have
\begin{equation*}
0=h_0'<h_{-n}'<H'\le y_0,\quad n\geq0,
\end{equation*}
and we get the lower bound for $y_0$. On the contrary, if \eqref{yy0r} holds, in particular we have that $h_{-n}<H'\leq y_0\leq H<h_n$ for every $n\geq0$. Following the same steps as before and Theorem 2.1 of \cite{GdI3}, using an argument of strong induction, will lead us to the fact that both $P_U$ and $P_L$ are stochastic matrices with the conditions that $0<x_n,y_n,s_n,r_n<1,n\in\ZZ$.
\end{proof}

Consider now the LU factorization of the stochastic matrix $P$ given by \eqref{QZ} in the following way
\begin{equation}\label{QZLU}
P=\left(\begin{array}{ccc|ccc}
\ddots&\ddots&&&\\
&\tilde r_{-1}&\tilde s_{-1}&0&\\
\hline
&&\tilde r_0&\tilde s_0&0\\
&&&\tilde r_1&\tilde s_1&0\\
&&&&\ddots&\ddots
\end{array}
\right)\left(\begin{array}{cc|cccc}
\ddots&\ddots&&\\
0&\tilde y_{-1}&\tilde x_{-1}&&\\
\hline
&0&\tilde y_0&\tilde x_0&\\
&&0&\tilde y_1&\tilde x_1\\
&&&&\ddots&\ddots
\end{array}
\right)=\widetilde P_L\widetilde P_U,
\end{equation}
where again $\widetilde P_L$ and $\widetilde P_U$ are also stochastic matrices, i.e. all entries of $\tilde P_L$ and $\tilde P_U$ are nonnegative and
\begin{equation}\label{stplu}
\tilde r_n+ \tilde s_n=1, \quad \tilde y_n +\tilde x_n=1,\quad  n \in \ZZ.
\end{equation}
Now, a direct computation shows that
\begin{align}
\nonumber a_n&=\tilde s_n \tilde x_n,\\
\label{LUd}b_n&=\tilde r_n\tilde x_{n-1}+\tilde s_n\tilde y_n,\quad n\in\ZZ,\\
\nonumber c_n&=\tilde r_n\tilde y_{n-1}.
\end{align}
By the irreducibility condition we have that $0<\tilde r_n$, $\tilde s_n$, $\tilde y_n$, $\tilde x_n<1$, $n\in \ZZ$. Now there is an \emph{important difference} if we compare with the case of random walks on $\ZZ_{\geq0}$, where the LU factorization is unique. In this case there will be also one \emph{free parameter}, namely $\tilde r_0$, from where we can compute all the coefficients of the matrices $\widetilde P_L$ and $\widetilde P_U$. If we fix $\tilde r_0$, for the nonnegative values of the indices we can compute $\tilde s_0$, $\tilde x_0$, $\tilde y_0$, $\tilde r_1$, $\tilde s_1$, $\tilde x_1$, $\tilde y_1$, $\tilde r_2$, $\dots$ recursively from \eqref{stplu} and \eqref{LUd}, while for the negative values of the indices we can compute $\tilde y_{-1}$, $\tilde x_{-1}$, $\tilde s_{-1}$, $\tilde r_{-1}$, $\tilde y_{-2}$, $\tilde x_{-2}$, $\tilde s_{-2}$, $\tilde r_{-2}$, $\dots$ recursively again from \eqref{stplu} and \eqref{LUd}.

Following the same steps as in Lemma 2.1 of \cite{GdI3} we have that for a decomposition like in \eqref{QZLU}, i.e. $P=\widetilde P_L \widetilde P_U$, the matrix $\widetilde P_L$ is stochastic if and only if the matrix $\widetilde P_U$ is stochastic and by \eqref{LUd}, the following relations holds
\begin{equation*}\label{relyslu}
\tilde r_n=\frac{c_n}{1-\tilde x_{n-1}},\quad \tilde x_{n}=\frac{a_{n}}{1-\tilde r_n},\quad n\in\ZZ.
\end{equation*}

Although we can compute all coefficients $\tilde r_n$, $\tilde s_n$, $\tilde y_n$ and $\tilde x_n$ in terms of one free parameter $\tilde r_0$, we can not infer anything about the positivity of these coefficients. However we have an analogue of Theorem \ref{thmfracUL} where now the free parameter $\tilde r_0$ will be bounded by the same continued fractions.

\begin{theorem}\label{thmfracLU}
Let $H$ and $H'$ be the continued fractions given by \eqref{cfUL} and the corresponding convergents $h_n$ and $h_{-n}$ defined by \eqref{hhh}. Assume conditions \eqref{AnBn}. Then both $H$ and $H'$ are convergent. Moreover, let $P=\widetilde P_L\widetilde P_U$ as in \eqref{QZLU}. Assume that $H'\leq H$. Then, both $\widetilde P_L$ and $\widetilde P_U$ are stochastic matrices if and only if we choose $\tilde r_0$ in the following range
\begin{equation*}\label{rr0r}
H'\leq \tilde r_0\leq H.
\end{equation*}
\end{theorem}
\begin{proof} 
The proof is similar to the proof of Theorem \ref{thmfracUL} but using \eqref{LUd} instead of \eqref{ULd}.
\end{proof}

\section{Stochastic Darboux transformations and the associated spectral matrices}\label{sec3}

In the previous section we have shown under what conditions a doubly infinite stochastic matrix $P$ like in \eqref{QZ} can be decomposed as a UL (or LU) factorization where both factors are still stochastic matrices. In both factorizations we have one free parameter. We knew about this fact for the UL factorization, but now for the LU factorization the phenomenon is new. Once we have a UL (or LU) factorization we can perform what is called a \emph{discrete Darboux transformation}. The Darboux transformation has a long history but probably the first reference of a discrete Darboux transformation like we study here appeared in \cite{MS} in connection with the Toda lattice. We explain now what a Darboux transformation is in our context.

If $P=P_UP_L$ as in \eqref{QZUL}, then by inverting the order of the factors we obtain another tridiagonal matrix of the form
\begin{equation}\label{QZULdar}
\widetilde P=P_LP_U=\left(\begin{array}{ccc|ccc}
\ddots&\ddots&&&\\
&r_{-1}&s_{-1}&0&\\
\hline
&&r_0&s_0&0\\
&&&r_1&s_1&0\\
&&&&\ddots&\ddots
\end{array}
\right)\left(\begin{array}{cc|cccc}
\ddots&\ddots&&\\
0&y_{-1}&x_{-1}&&\\
\hline
&0&y_0&x_0&\\
&&0&y_1&x_1\\
&&&&\ddots&\ddots
\end{array}
\right).
\end{equation}
The new coefficients of the matrix $\widetilde P$ are given by
\begin{align}
\nonumber\tilde{a}_n&=s_nx_{n},\\
\label{coeffDLU}\tilde{b}_n&=r_nx_{n-1}+s_ny_n,\quad n\in\ZZ,\\
\nonumber\tilde{c}_n&=r_ny_{n-1}.
\end{align}
The matrix $\widetilde{P}$ is actually stochastic, since the multiplication of two stochastic matrices is again a stochastic matrix. Therefore it gives a \emph{family} of new random walks $\{\widetilde X_t : t=0,1,\ldots\}$ on the integers $\ZZ$ with coefficients $(\tilde{a}_n)_{n\in\ZZ}$, $(\tilde{b}_n)_{n\in\ZZ}$ and $(\tilde{c}_n)_{n\in\ZZ}$ depending on a free parameter $y_0$. 

The same can be done for the LU factorization \eqref{QZLU} of the form $P=\widetilde P_L\widetilde P_U$. Indeed, the new random walk is given by
\begin{equation}\label{QZLUdar}
\widehat P=\widetilde P_U\widetilde P_L=\left(\begin{array}{cc|cccc}
\ddots&\ddots&&\\
0&\tilde y_{-1}&\tilde x_{-1}&&\\
\hline
&0&\tilde y_0&\tilde x_0&\\
&&0&\tilde y_1&\tilde x_1\\
&&&&\ddots&\ddots
\end{array}
\right)\left(\begin{array}{ccc|ccc}
\ddots&\ddots&&&\\
&\tilde r_{-1}&\tilde s_{-1}&0&\\
\hline
&&\tilde r_0&\tilde s_0&0\\
&&&\tilde r_1&\tilde s_1&0\\
&&&&\ddots&\ddots
\end{array}
\right),
\end{equation}
where the new coefficients are
\begin{equation}\label{coeffDUL}
\begin{split}
\hat{a}_n&=\tilde x_{n}\tilde s_{n+1},\\
\hat{b}_n&=\tilde x_{n}\tilde r_{n+1}+\tilde y_n\tilde s_n,\quad n\in\ZZ,\\
\hat{c}_n&=\tilde y_{n}\tilde r_n.
\end{split}
\end{equation}
Again, the matrix $\widehat{P}$ is stochastic, so we have a \emph{family} of new random walks $\{\widehat X_t : t=0,1,\ldots\}$ on the integers $\ZZ$ with coefficients $(\hat{a}_n)_{n\in\ZZ}$, $(\hat{b}_n)_{n\in\ZZ}$ and $(\hat{c}_n)_{n\in\ZZ}$ depending on one free parameter $\tilde r_0$. 

In terms of a model driven by urn experiments both factorizations may be thought as two urn experiments, Experiment 1 and Experiment 2, respectively. We first perform the Experiment 1 and with the result we immediately perform the Experiment 2. The model for the Darboux transformation will be reversing the order of both experiments. For more details about these urn models see \cite{GdI3,GdI4}.

\medskip

Now we will focus in the following question: given the spectrum of the doubly infinite matrix $P$, how can we compute the spectrum of the Darboux transformations $\widetilde{P}$ and $\widehat{P}$? It turns out that both transformations will be related with what is called a \emph{Geronimus transformation} (see below). Before that let us introduce some notation to study the spectral measures associated with the original random walk $P$.

We will follow the last section of \cite{KMc6}. For $P$ like in \eqref{QZ} consider the eigenvalue equation $ x q^\alpha(x) = P q^\alpha(x) $ where $q^\alpha(x)=(\cdots, Q_{-1}^\alpha(x),Q_0^\alpha(x),Q_1^\alpha(x),\cdots)^T, \alpha=1,2$. For each $ x $ real or complex there exist two polynomial families of linearly independent solutions $ Q_n ^ {\alpha} (x), \alpha = 1,2, n \in \ZZ , $ depending on the initial values at $n = 0 $ and $ n = -1 $. These polynomials are given by
\begin{align}
\nonumber Q_0^1(x)&=1,\quad Q_{0}^2(x)=0,\\
\label{TTRRZ}Q_{-1}^1(x)&=0,\quad Q_{-1}^2(x)=1,\\
\nonumber xQ_n^{\alpha}(x)&=a_nQ_{n+1}^{\alpha}(x)+b_nQ_n^{\alpha}(x)+c_nQ_{n-1}^{\alpha}(x),\quad n\in\ZZ,\quad\alpha=1,2.
\end{align}
Observe that 
\begin{equation}\label{degqs}
\begin{split}
\deg(Q_n^1)&=n,\quad n\geq0,\hspace{.95cm}\deg(Q_n^2)=n-1,\quad n\geq1,\\
\deg(Q_{-n-1}^1)&=n-1,\quad n\geq1,\quad\deg(Q_{-n-1}^2)=n,\quad n\geq0.
\end{split}
\end{equation}
From the three-term recurrence relation \eqref{TTRRZ} it is possible to compute the leading coefficients of the polynomials $Q_n^{\alpha}(x), \alpha =1,2, n\in \ZZ$. Indeed, for $n\geq0,$ we have
\begin{align}
\label{lead1}Q_n^1(x)&=R_n^1x^n+\mathcal{O}(x^{n-1}),\quad R_0^1=1,\quad R_n^1=(a_0\cdots a_{n-1})^{-1},\quad n\geq1,\\
\nonumber Q_{-n-1}^1(x)&=L_{n-1}^1x^{n-1}+\mathcal{O}(x^{n-2}),\quad L_{n-1}^1=-a_{-1}(c_{-1}\cdots c_{-n})^{-1},\quad n\geq1,
\end{align}
and
\begin{align}
\nonumber Q_n^2(x)&=R_{n-1}^2x^{n-1}+\mathcal{O}(x^{n-2}),\quad R_{n-1}^2=-c_0(a_0\cdots a_{n-1})^{-1},\quad n\geq1,\\
\label{lead2}Q_{-n-1}^2(x)&=L_{n}^2x^n+\mathcal{O}(x^{n-1}),\quad L_0^2=1,\quad L_{n}^2=(c_{-1}\cdots c_{-n})^{-1},\quad n\geq1.
\end{align}
Let us define the \emph{potential coefficients} as
\begin{equation}\label{potcoeff}
\pi_0=1,\quad \pi_n=\frac{a_0a_1\cdots a_{n-1}}{c_1c_2\cdots c_n},\quad \pi_{-n}=\frac{c_0c_{-1}\cdots c_{-n+1}}{a_{-1}a_{-2}\cdots a_{-n}},\quad n\geq1.
\end{equation}
In particular we have that $\pi P=\pi$, i.e. $\pi=(\pi_n)_{n\in\ZZ}$ is an invariant vector of $P$. In the Hilbert space $\ell^2_{\pi}(\ZZ)$ the matrix $P$ gives rise to a self-adjoint operator of norm $\leq1$, which we will denote by $P$, abusing the notation. This result is a consequence of Corollary 2.2 of \cite{MR} since the coefficients $(a_n)_{n\in\ZZ}$ and $(c_n)_{n\in\ZZ}$ are probabilities and therefore $0<a_n,c_n<1$. Applying the spectral theorem \emph{three times}, there exist three unique measures $\psi_{11}(x), \psi_{22}(x)$ and $\psi_{12}(x)$  (since $\psi_{12}(x)=\psi_{21}(x)$ as a consequence of $P$ being self-adjoint and the symmetry of the inner product) supported on the interval $[-1,1]$ such that 
\begin{equation}\label{ortoZ}
\sum_{\alpha,\beta=1}^{2}\int_{-1}^{1}Q_i^{\alpha}(x)Q_j^{\beta}(x)d\psi_{\alpha\beta}(x)=\frac{\delta_{i,j}}{\pi_j},\quad i,j\in\ZZ.
\end{equation}

The measures $\psi_{11}$ and $\psi_{22}$ are positive (in fact $\psi_{11}$ is a probability measure but $\psi_{22}$ is not, since $\int_{-1}^1d\psi_{22}(x)=1/\pi_{-1}$). The measure $\psi_{12}$ is a signed measure satisfying $0=\int_{-1}^1d\psi_{12}(x)$. For simplicity let us assume that the three measures are continuously differentiable with respect to the Lebesgue measure, i.e. $d\psi_{\alpha\beta}(x)=\psi_{\alpha\beta}(x)dx, \alpha,\beta=1,2$, abusing the notation. These 3 measures can be written in matrix form as the $2\times2$ matrix
\begin{equation}\label{2spmt}
\Psi(x)=\begin{pmatrix} \psi_{11}(x) & \psi_{12}(x)\\ \psi_{12}(x) & \psi_{22}(x) \end{pmatrix},
\end{equation}
so that the orthogonality relations \eqref{ortoZ} can be written in matrix form as
\begin{equation}\label{ortoZ2}
\int_{-1}^{1}\left(Q_i^1(x),Q_i^2(x)\right)\Psi(x)\begin{pmatrix}Q_j^1(x)\\Q_j^2(x)\end{pmatrix}dx=\frac{\delta_{i,j}}{\pi_j},\quad i,j\in\ZZ.
\end{equation}
The matrix $\Psi(x)$ in \eqref{2spmt} is called the \emph{spectral matrix} associated with $P$. The orthogonality conditions are valid for any indexes $ i, j \in \ZZ$.  With this information we can compute the $n$-step transition probabilities of the random walk $\{X_t : t=0,1,\ldots\}$, given by the so-called Karlin-McGregor integral representation formula (see \cite{KMc6})
\begin{equation}\label{KmcG1}
P_{ij}^{(n)}\doteq\mathbb{P}(X_n=j \;| X_0=i)=\pi_j\int_{-1}^{1}x^n\left(Q_i^1(x),Q_i^2(x)\right)\Psi(x)\begin{pmatrix}Q_j^1(x)\\Q_j^2(x)\end{pmatrix}dx,\quad i,j\in\ZZ.
\end{equation}
It is possible to relabel the states in such a way that all the information of $ P $ can be collected in a semi-infinite block tridiagonal matrix $ \bm P $ with blocks of size $ 2 \times2 $. Indeed, after the new labeling
$$
\{0,1,2,\ldots\}\to\{0,2,4,\ldots\},\quad\mbox{and}\quad\{-1,-2,-3,\ldots\}\to\{1,3,5,\ldots\},
$$
we have that $P$ (doubly infinite tridiagonal) is equivalent to a matrix $\bm P$ (semi-infinite block tridiagonal) of the form
\begin{equation*}
\bm P=\begin{pmatrix}
B_0&A_0&&&\\
C_1&B_1&A_1&&\\
&C_2&B_2&A_2&\\
&&\ddots&\ddots&\ddots
\end{pmatrix},
\end{equation*}
where
\begin{align*}
B_0&=\begin{pmatrix} b_0 & c_0\\a_{-1} & b_{-1}\end{pmatrix},\quad B_n=\begin{pmatrix} b_n & 0\\ 0 & b_{-n-1}\end{pmatrix},\quad n\geq1,\\
A_n&=\begin{pmatrix} a_n & 0\\ 0 & c_{-n-1}\end{pmatrix},\quad n\geq0,\quad C_n=\begin{pmatrix} c_n & 0\\ 0 & a_{-n-1}\end{pmatrix},\quad n\geq1.
\end{align*}
The random walk generated by $\bm P $ can be interpreted as a walk that takes values in the two-dimensional state space $\ZZ_{\geq0}\times \{1,2 \}. $ These type of processes are called discrete-time \emph {quasi-birth-and-death processes}. In general these processes allow transitions between all two-dimensional adjacent states (see \cite{LaR, Neu} for a general reference).

If we define the matrix-valued polynomials
\begin{equation}\label{2QMM}
\bm Q_n(x)=\begin{pmatrix} Q_n^1(x) & Q_n^2(x) \\ Q_{-n-1}^1(x) & Q_{-n-1}^2(x)\end{pmatrix},\quad n\geq0,
\end{equation}
then we have
\begin{align*}
x\bm Q_0(x)&=A_0\bm Q_{1}(x)+B_0\bm Q_0(x),\quad \bm Q_0(x)=I_{2\times2},\\
x\bm Q_n(x)&=A_n\bm Q_{n+1}(x)+B_n\bm Q_n(x)+C_n\bm Q_{n-1}(x),\quad n\geq1,
\end{align*} 
where $I_{2\times2}$ denotes the $2\times2$ identity matrix. Observe that $\deg(\bm Q_n)=n$ and the leading coefficient is a nonsingular matrix (by \eqref{lead1} and \eqref{lead2}). The matrix orthogonality is defined in terms of the (matrix-valued) inner product
\begin{equation}\label{ortoZ3}
\int_{-1}^{1}\bm Q_n(x)\Psi(x)\bm Q_m^*(x)dx=\begin{pmatrix} 1/\pi_n & 0 \\ 0 & 1/\pi_{-n-1}\end{pmatrix}\delta_{nm},
\end{equation}
where $A^*$ is the Hermitian transpose of a matrix $A$ and $\pi=(\pi_n)_{n\in\ZZ}$ is given by \eqref{potcoeff}. In this case we have (see \cite{DRSZ, G2}) the Karlin-McGregor integral representation formula where the $2\times2$ block entry $(i,j)$ is given by
\begin{equation*}
\bm P_{ij}^{(n)}=\left(\int_{-1}^{1}x^n\bm Q_i(x)\Psi(x)\bm Q_j^*(x)dx\right)\begin{pmatrix} \pi_j & 0 \\ 0 & \pi_{-j-1}\end{pmatrix},\quad i,j\in\ZZ_{\geq0}.
\end{equation*}
\begin{remark}\label{remrec}
In Corollaries 4.1 and 4.2 of \cite{DRSZ} one can find some results concerning recurrence for discrete-time quasi-birth-and-death processes. Applying these to the case of random walks on $\ZZ$ we have that the random walk is \emph{recurrent} if and only if $\int_{-1}^1\psi_{\alpha,\beta}(x)/(1-x)dx=\infty$ for some $\alpha,\beta=1,2,$ and it is \emph{positive recurrent} if and only if one of the measures $\psi_{\alpha,\beta}(x),\alpha,\beta=1,2,$ has a jump at the point 1.
\end{remark}

In the following lemma we will give a characterization of the orthogonality of the vector-valued polynomials $\left(Q_n^1(x),Q_n^2(x)\right)$ in terms of monomials.
\begin{lemma}\label{lemorto}
Let $(Q_n^\alpha)_{n\in\ZZ}$ be the polynomials defined by \eqref{TTRRZ}. Then the vector-valued polynomials $\left(Q_n^1(x),Q_n^2(x)\right)$, $n\in\ZZ$ are orthogonal in the sense of \eqref{ortoZ2} if and only if for $n\geq0$ we have
\begin{equation}\label{cond1}
\int_{-1}^1\left(Q_n^1(x),Q_n^2(x)\right)\Psi(x)x^jdx=
\begin{cases}
(0,0),&\mbox{for}\quad j=0,1,\ldots,n-1,\\
(\alpha_n,0),\alpha_n\neq0,&\mbox{for}\quad j=n,
\end{cases}
\end{equation}
and
\begin{equation}\label{cond2}
\int_{-1}^1\left(Q_{-n-1}^1(x),Q_{-n-1}^2(x)\right)\Psi(x)x^jdx=
\begin{cases}
(0,0),&\mbox{for}\quad j=0,1,\ldots,n-1,\\
(0,\beta_n),\beta_n\neq0,&\mbox{for}\quad j=n.
\end{cases}
\end{equation}
Moreover $\alpha_0=1, \alpha_n=c_1\cdots c_n, n\geq1$ and $\beta_n=c_0^{-1}a_{-1}\cdots a_{-n-1},n\geq0$.
\end{lemma}
\begin{proof}
The orthogonality conditions \eqref{ortoZ2} are equivalent to the matrix orthogonality \eqref{ortoZ3}. Since $\bm Q_n(x)$ in \eqref{2QMM} is a matrix polynomial of degree $n$ with nonsingular leading coefficient the orthogonality is equivalent to $\int_{-1}^1\bm Q_n(x)\Psi(x) x^jdx=0_{2\times2}$ for $j=0,1,\ldots,n-1,$ where $0_{2\times2}$ denotes the $2\times2$ null matrix, and $\int_{-1}^1\bm Q_n(x)\Psi(x) x^ndx$ is a nonsingular (diagonal) matrix. Looking at the rows of these expressions we get \eqref{cond1} and \eqref{cond2}. The values of $\alpha_n$ and $\beta_n$ can be computed using \eqref{lead1}, \eqref{lead2} and \eqref{potcoeff}.
\end{proof}

Now that we have studied the spectral properties of the doubly infinite matrix $P$ in \eqref{QZ} let us study the spectral matrices associated with the Darboux transformations $\widetilde P$ in \eqref{QZULdar} and $\widehat P$ in \eqref{QZLUdar}. We will analyze both cases separately.

\subsection{Darboux transformation for the UL case}

Consider the discrete Darboux transformation $\widetilde P$ in \eqref{QZULdar} with probability coefficients $(\tilde{a}_n)_{n\in\ZZ}$, $(\tilde{b}_n)_{n\in\ZZ}$ and $(\tilde{c}_n)_{n\in\ZZ}$ given by \eqref{coeffDLU}. Before defining the corresponding polynomials associated with $\widetilde P$ let us introduce an auxiliary family of polynomials $S_n^\alpha(x)$ given by the relation $s^\alpha(x)=P_Lq^\alpha(x)$, where $q^\alpha(x)=(\cdots, Q_{-1}^\alpha(x),Q_0^\alpha(x),Q_1^\alpha(x),\cdots)^T,$ and $s^\alpha(x)=(\cdots, S_{-1}^\alpha(x), S_0^\alpha(x),S_1^\alpha(x),\cdots)^T, \alpha=1,2$, i.e. 
\begin{equation}\label{qenqhat}
S_n^\alpha(x)=s_n Q_n^\alpha (x)+ r_n Q_{n-1}^\alpha(x),\quad n \in \mathbb{Z}, \quad \alpha=1,2.
\end{equation}
From the UL factorization we also have $P_U s^\alpha(x)= xq^\alpha(x)$, that is
\begin{equation}\label{qen1}
xQ_n^\alpha (x)=x_n S_{n+1}^\alpha(x)+y_n S_n^\alpha(x), \quad n\in \mathbb{Z}, \quad\alpha=1,2.
\end{equation}
Evaluating \eqref{qen1} at $x=0$ we get recursively
\begin{equation}
\begin{split}\label{q0hat}
S_{n}^\alpha (0)&=(-1)^{n}\frac{y_0\dots y_{n-1}}{x_0 \dots x_{n-1}}S_0^\alpha (0), \quad n\ge 1,\\
S_{-n-1}^\alpha (0)&=(-1)^{n+1} \frac{x_{-1}\dots x_{-n-1}}{y_{-1}\dots y_{-n-1}} S_0^\alpha(0), \quad n\ge 0,
\end{split}
\end{equation}
where 
\begin{equation}\label{relacons}
S_0^\alpha(0)=\begin{cases} s_0, \quad \text{if}\quad \alpha=1, \\ r_0, \quad \text{if} \quad  \alpha=2. \end{cases}
\end{equation}
The equations \eqref{q0hat} establish a direct relation between the polynomials $(S_n^\alpha)_{n\in\ZZ}, \alpha=1,2,$ given by
\begin{equation} \label{relq0}
s_0 S_n^2(0)=r_0 S_n^1(0), \quad n\in \mathbb{Z}.
\end{equation}
Another useful relation follows using \eqref{qen1} and \eqref{q0hat} which gives the polynomials $(S_n^\alpha)_{n\in \mathbb{Z}}$ in terms of $(Q_n^\alpha)_{n\in \mathbb{Z}}$. Indeed, for $n\geq0,$
\begin{equation*} 
\begin{split}
S_{n+1}^\alpha(x)&=\frac{x}{x_n} Q_n^\alpha(x)-\frac{y_n}{x_n}S_n^\alpha (x)= \frac{x}{x_{n}}Q_n^\alpha(x) +\frac{S_{n+1}^\alpha(0)}{ S_{n}^\alpha(0)}\left[ \frac{x}{x_{n-1}} Q_{n-1}^\alpha(x)+\frac{S_{n}^\alpha(0)}{ S_{n-1}^\alpha(0)}S_{n-1}^\alpha (x) \right]\\
&= x\left[ \frac{Q_n^\alpha(x)}{x_n}+ \frac{S_{n+1}^\alpha(0)}{ S_{n}^\alpha(0)}  \frac{Q_{n-1}^\alpha(x)}{x_{n-1}} \right] + \frac{S_{n+1}^\alpha(0)}{ S_{n-1}^\alpha(0)} S_{n-1}^\alpha(x)= \cdots =x\sum_{j=0}^{n} \frac{S_{n+1}^\alpha(0)}{ S_{j+1}^\alpha(0)}\frac{Q_j^\alpha(x)}{x_j} +S_{n+1}^\alpha(0),
\end{split}
\end{equation*}
since $S_0^\alpha(x)$ is constant (see \eqref{relacons}). Therefore
\begin{equation} \label{qhatrec1}
\begin{split}
S_n^\alpha(x)&=S_{n}^\alpha(0) \left[ 1+x\sum_{j=0}^{n-1} \frac{Q_j^\alpha(x) }{ S_{j+1}^\alpha(0) x_j} \right], \quad n\ge 1.
\end{split}
\end{equation}
Similarly
\begin{equation} \label{qhatrec2}
\begin{split}
S_{-n-1}^\alpha(x)&=S_{-n-1}^\alpha(0) \left[ 1+x\sum_{j=0}^{n} \frac{Q_{-j-1}^\alpha(x) }{ S_{-j-1}^\alpha(0) y_{-j-1}} \right], \quad n\ge 0.
\end{split}
\end{equation}
Observe that this auxiliary family $(S_n^\alpha)_{n\in \mathbb{Z}}$ does not satisfy the same initial conditions as the family $(Q_n^\alpha)_{n\in \mathbb{Z}}$ since by \eqref{qenqhat} we have
\begin{equation*}
\begin{split}
S_0^1(x)=&s_0, \qquad\qquad S_0^2(x)=r_0,\\
S_{-1}^1(x)=&-\frac{x_{-1}s_0}{y_{-1}}, \quad S_{-1}^2(x)=\frac{x-x_{-1}r_0}{y_{-1}}.
\end{split}
\end{equation*}
The degrees of the polynomials $(S_n^\alpha)_{n\in \mathbb{Z}}$ are also not the same as the degrees of the polynomials $(Q_n^\alpha)_{n\in \mathbb{Z}}$, since
\begin{equation}\label{degqss}
\begin{split}
\deg(S_n^1)&=n, \quad  n\ge 0, \quad \deg(S_n^2)=n,\quad n\ge 0,\\
\deg(S_{-n-1}^1)&=n, \quad n\ge 0, \quad \deg(S_{-n-1}^2)=n+1,\quad n\ge 0.
\end{split}
\end{equation}
Therefore the family of matrix polynomials
\begin{equation*}\label{S00}
\bm S_n(x)=\begin{pmatrix} S_n^1(x)&S_n^2(x) \\ S_{-n-1}^1(x) & S_{-n-1}^2(x)\end{pmatrix}, \quad n\geq0,
\end{equation*}
has degree $n+1$ and singular leading coefficient. Now we will define a new family of polynomials which will turn out to be the associated family of the Darboux transformation $\widetilde P$. For $n\geq0$ define 
\begin{equation}\label{qtilde}
\widetilde{\bm Q}_n(x)=\bm S_n(x) \bm S_0^{-1}(x),\quad n\geq0,
\end{equation}
where
\begin{equation}\label{S000}
\bm S_0(x)=\begin{pmatrix} s_0& r_0 \\ -\D\frac{x_{-1}s_0}{y_{-1}}& \D\frac{x-x_{-1}r_0}{y_{-1}}
\end{pmatrix}.
\end{equation}
Following the same representation as in \eqref{2QMM} we can define the functions $(\widetilde Q_n^\alpha)_{n\in \mathbb{Z}}, \alpha=1,2$, which turn out to be polynomials, as the following proposition shows.
\begin{proposition}\label{proppol}
Let $\widetilde{\bm Q}_n(x), n\geq0,$ be the matrix function defined by \eqref{qtilde}. Then, for $n\geq0$, $\widetilde{\bm Q}_n(x)$ is a matrix polynomial of degree exactly $n$ with nonsingular leading coefficient and $\widetilde{\bm Q}_0(x)=I_{2\times2}$.
\end{proposition}
\begin{proof}
Computing the inverse of $\bm S_0(x)$ we have
$$
\bm S_0^{-1}(x)=\frac{1}{x}\begin{pmatrix} \D\frac{x-x_{-1}r_0}{s_{0}}& -\D\frac{y_{-1}r_0}{s_{0}} \\ x_{-1}& y_{-1}\end{pmatrix}.
$$
Observe that $|\bm S_0(x)|=\frac{xs_0}{y_{-1}}$, so the inverse is not well-defined at $x=0$. We will show that we can avoid this problem using the properties of the polynomials $(S_n^\alpha)_{n\in \mathbb{Z}}$. Indeed, from \eqref{qtilde} we have
\begin{equation}\label{qtilss}
\begin{split}
\widetilde Q_n^1(x)&=\frac{S_n^1(x)}{s_0}+\frac{x_{-1}}{xs_0}\left(s_0S_n^2(x)-r_0S_n^1(x)\right),\quad n\in\ZZ,\\
\widetilde Q_n^2(x)&=\frac{y_{-1}}{xs_0}\left(s_0S_n^2(x)-r_0S_n^1(x)\right),\quad n\in\ZZ.
\end{split}
\end{equation}
A straightforward computation using \eqref{qhatrec1}, \eqref{qhatrec2} and \eqref{relq0} gives
\begin{align*}
s_0S_0^2(x)-r_0S_0^1(x)&=0,\\
s_0S_n^2(x)-r_0S_n^1(x)&=x\sum_{j=0}^{n-1}\frac{1}{x_j}\left(s_0\frac{S_n^2(0)Q_j^2(x)}{S_{j+1}^2(0)}-r_0\frac{S_n^1(0)Q_j^1(x)}{S_{j+1}^1(0)}\right),\quad n\geq1,\\
s_0S_{-n-1}^2(x)-r_0S_{-n-1}^1(x)&=x\sum_{j=0}^{n}\frac{1}{y_{-j-1}}\left(s_0\frac{S_{-n-1}^2(0)Q_{-j-1}^2(x)}{S_{-j-1}^2(0)}-r_0\frac{S_{-n-1}^1(0)Q_{-j-1}^1(x)}{S_{-j-1}^1(0)}\right),\quad n\geq0.
\end{align*}
Therefore from these relations we can see that $(\widetilde Q_n^\alpha)_{n\in \mathbb{Z}}, \alpha=1,2$ are indeed polynomials. A close look to the degrees of $(Q_n^\alpha)_{n\in \mathbb{Z}}, \alpha=1,2,$ in \eqref{degqs} gives that $\deg\left((s_0S_n^2(x)-r_0S_n^1(x))/x\right)=n-1,n\geq1$. Therefore from \eqref{degqss} we get $\deg(\widetilde Q_n^1)=n, n\geq0$ and $\deg(\widetilde Q_n^2)=n-1,n\geq1$. On the other hand $\deg\left((s_0S_{-n-1}^2(x)-r_0S_{-n-1}^1(x))/x\right)=n,n\geq0$. Therefore $\deg(\widetilde Q_{-n-1}^2)=n, n\geq0$. Finally, $\widetilde Q_{-n-1}^1$ is in principle a polynomial of degree $n$, but we will see that in fact is a polynomial of degree $n-1$. Indeed, call $\Lambda_n$ the coefficient of $x^n$ in $\widetilde Q_{-n-1}^1$. Then, using \eqref{lead1}, \eqref{lead2}, \eqref{qenqhat} and \eqref{ULd}, we have
$$
\Lambda_n=\frac{1}{s_0}\left(-\frac{a_{-1}r_{-n-1}}{c_{-1}\cdots c_{-n-1}}+\frac{x_{-1}s_0}{y_{-n-1}c_{-1}\cdots c_{-n}}\right)=\frac{1}{s_0c_{-1}\cdots c_{-n}}\left(-\frac{a_{-1}r_{-n-1}}{c_{-n-1}}+\frac{x_{-1}s_0}{y_{-n-1}}\right)=0.
$$
Therefore $\deg(\widetilde Q_{-n-1}^1)=n-1, n\geq1$. The fact that $\widetilde{\bm Q}_0(x)=I_{2\times2}$ comes from the definition \eqref{qtilde}.
\end{proof}
\begin{remark}
The previous proposition shows that the \emph{Darboux polynomials} $(\widetilde Q_n^\alpha)_{n\in \mathbb{Z}}, \alpha=1,2,$ satisfy the same initial conditions and degree conditions than the original polynomials $(Q_n^\alpha)_{n\in \mathbb{Z}}, \alpha=1,2$. They also satisfy the three-term recurrence relation
\begin{equation}\label{TTRRZUL}
\begin{split}
\widetilde Q_0^1(x)&=1,\quad \widetilde Q_{0}^2(x)=0,\\
\widetilde Q_{-1}^1(x)&=0,\quad \widetilde Q_{-1}^2(x)=1,\\
x\widetilde Q_n^{\alpha}(x)&=\tilde a_n\widetilde Q_{n+1}^{\alpha}(x)+\tilde b_n\widetilde Q_n^{\alpha}(x)+\tilde c_n\widetilde Q_{n-1}^{\alpha}(x),\quad n\in\ZZ,\quad\alpha=1,2,
\end{split}
\end{equation}
where the Darboux coefficients $(\tilde{a}_n)_{n\in\ZZ}$, $(\tilde{b}_n)_{n\in\ZZ}$ and $(\tilde{c}_n)_{n\in\ZZ}$ are defined by \eqref{coeffDLU}. This is a consequence of writing the polynomials $(\widetilde Q_n^\alpha)_{n\in \mathbb{Z}}, \alpha=1,2,$ in terms of $(S_n^\alpha)_{n\in \mathbb{Z}}, \alpha=1,2,$ (see \eqref{qtilss}) and the fact that the polynomials $(S_n^\alpha)_{n\in \mathbb{Z}}, \alpha=1,2,$ satisfy the same three-term recurrence relation (with different initial conditions) by construction (see \eqref{qenqhat}).
\end{remark}

We will now show one the the main results of this paper, namely how to compute the spectral matrix associated with the Darboux random walk  $\widetilde P$ in \eqref{QZULdar} and prove that $(\widetilde Q_n^\alpha)_{n\in \mathbb{Z}}, \alpha=1,2,$ are the corresponding orthogonal polynomials. We define first the potential coefficients associated with $\widetilde P$ given by
\begin{equation}\label{cofpotUL}
\tilde \pi_0=1, \quad \tilde \pi_n=\frac{\tilde a_0 \cdots \tilde a_{n-1}}{\tilde c_1 \cdots \tilde c_{n}}, \quad \tilde \pi_{-n}=\frac{\tilde c_0 \cdots \tilde c_{-n+1}}{ \tilde a_{-1} \cdots \tilde a_{-n}}, \quad n\ge 1.
\end{equation}
\begin{theorem}\label{thmorto}
Let $\{X_t: t=0, 1, \dots\}$ be the random walk on $\ZZ$ with transition probability matrix $P$ given by \eqref{QZ} and $\{\widetilde X_t : t=0, 1, \dots\}$ the Darboux random walk on $\ZZ$ with transition probability matrix $\widetilde P$ given by \eqref{QZULdar}. Assume that $M_{-1}=\D\int_{-1}^{1} \D\frac{\Psi(x)}{x}dx$ is well-defined (entry by entry), where $\Psi(x)$ is the original spectral matrix (see \eqref{2spmt}). Then the polynomials $(\widetilde Q_n^\alpha)_{n\in \mathbb{Z}}, \alpha=1,2,$ defined by \eqref{TTRRZUL} (see also \eqref{qtilde}) are orthogonal with respect to the following spectral matrix 
\begin{equation}\label{esptil}
\widetilde \Psi (x)=\bm S_0(x) \Psi_S (x) \bm S_0^*(x),
\end{equation}
where $\bm S_0(x)$ is defined by \eqref{S000} and
\begin{equation}\label{esphat}
\Psi_S(x)=\frac{y_0}{s_0}\frac{\Psi(x)}{x}+ \left[ \begin{pmatrix}1/s_0&0 \\ 0 &1/r_0 \end{pmatrix}- \frac{y_0}{s_0}M_{-1}\right] \delta_0(x).
\end{equation}
Moreover, we have
\begin{equation}\label{qusnorms}
\int_{-1}^{1} \widetilde{\bm Q}_n (x) \widetilde \Psi (x) \widetilde{\bm Q}_m^*(x) dx=\begin{pmatrix} 1/\tilde \pi_n &0\\ 0& 1/\tilde \pi_{-n-1}\end{pmatrix} \delta_{n,m},
\end{equation}
where $(\tilde \pi_n)_{n\in \mathbb{Z}}$ are the potential coefficients defined by \eqref{cofpotUL}.
\end{theorem}
\begin{proof}
Let $(Q_n^\alpha)_{n\in \mathbb{Z}}, \alpha=1,2,$ be the polynomials defined by \eqref{TTRRZ}, which are orthogonal with respect to the original spectral matrix $\Psi$. By Lemma \ref{lemorto} we have the orthogonality conditions \eqref{cond1} and \eqref{cond2}. Since $(\widetilde Q_n^\alpha)_{n\in \mathbb{Z}}, \alpha=1,2,$ satisfies the same initial and degree conditions than $(Q_n^\alpha)_{n\in \mathbb{Z}}, \alpha=1,2,$ we will use Lemma \ref{lemorto} to prove that $(\widetilde Q_n^\alpha)_{n\in \mathbb{Z}}, \alpha=1,2,$ are orthogonal with respect to $\widetilde \Psi (x)$ in \eqref{esptil}.

Assume first that $n\geq1$. Then we have, using \eqref{esptil}, \eqref{qtilde} and \eqref{esphat}, that
\begin{equation*}
\begin{split}
\int_{-1}^{1} \left(\widetilde Q_n^1(x), \widetilde Q_n^2(x)\right) \widetilde \Psi  (x)x^jdx  &= \int_{-1}^{1} \left(\widetilde Q_n^1(x), \widetilde Q_n^2(x)\right)  x\widetilde \Psi(x) x^{j-1}dx\\
&= \int_{-1}^{1} \left(\widetilde Q_n^1(x), \widetilde Q_n^2(x)\right) \bm S_0(x)  x\Psi_S(x)  \bm S_0^*(x) x^{j-1}dx\\
&=\frac{y_0}{s_0} \int_{-1}^{1} \left(S_n^1(x), S_n^2(x)\right) \bm S_0^{-1}(x) \bm S_0(x) \Psi(x)  \bm S_0^*(x) x^{j-1}dx \\
&=\frac{y_0}{s_0}  \int_{-1}^{1} \left(S_n^1(x), S_n^2(x)\right) \Psi(x)  \bm S_0^*(x) x^{j-1}dx. \\
\end{split}
\end{equation*}
Now, using \eqref{qenqhat}, the above expression can be written as
\begin{equation*}
\begin{split}
\int_{-1}^{1} \left(\widetilde Q_n^1(x), \widetilde Q_n^2(x)\right) \widetilde \Psi  (x)x^jdx&= \frac{s_n y_0}{s_0} \int_{-1}^{1} \left(Q_n^1(x), Q_n^2(x)\right) \Psi(x)  \bm S_0^*(x) x^{j-1}dx \\
&\quad+ \frac{r_ n y_0}{s_0}  \int_{-1}^{1} \left(Q_{n-1}^1(x), Q_{n-1}^2(x)\right) \Psi(x)  \bm S_0^*(x) x^{j-1}dx.
\end{split}
\end{equation*}
Writing $\bm S_0(x)=A+ xB$, where $A$ and $B$ are given by
\begin{equation} \label{q0ext}
A=\begin{pmatrix} s_0& r_0\\ -\D\frac{x_{-1}s_0}{y_{-1}} & -\D\frac{x_{-1}r_0}{y_{-1}}\end{pmatrix},\quad B=\begin{pmatrix} 0& 0\\ 0 & \D\frac{1}{y_{-1}}\end{pmatrix},
\end{equation}
the above expression can be written as
\begin{equation*}
\begin{split}
\int_{-1}^{1} \left(\widetilde Q_n^1(x), \widetilde Q_n^2(x)\right)& \widetilde \Psi  (x)x^jdx= \frac{s_n y_0}{s_0}\left[\int_{-1}^{1} \left(Q_n^1(x), Q_n^2(x)\right) \Psi(x)  A^* x^{j-1}dx+\int_{-1}^{1} \left(Q_n^1(x), Q_n^2(x)\right) \Psi(x)  B^* x^{j}dx\right]\\
&\quad+ \frac{r_ n y_0}{s_0}\left[\int_{-1}^{1} \left(Q_{n-1}^1(x), Q_{n-1}^2(x)\right) \Psi(x)  A^* x^{j-1}dx+\int_{-1}^{1} \left(Q_{n-1}^1(x), Q_{n-1}^2(x)\right) \Psi(x)  B^* x^{j}dx\right].
\end{split}
\end{equation*}
Using \eqref{cond1} we have that the first term of the sum vanishes for $j=1,\ldots,n$, the second term vanishes for $j=0,\ldots,n-1$, the third term vanishes for $j=1,\ldots,n-1$ and the fourth term vanishes for $j=0,\ldots,n-2$. Therefore the above expression vanishes for $j=1,\ldots,n-2$. For $j=n-1$ the only term that does not vanish is the fourth one. But in this case we have, using \eqref{cond1} and \eqref{q0ext}, that
\begin{equation}\label{Bbss}
\begin{split}
\int_{-1}^{1} \left(\widetilde Q_n^1(x), \widetilde Q_n^2(x)\right) \widetilde \Psi  (x)x^{n-1}dx&= \frac{r_ n y_0}{s_0}\int_{-1}^{1} \left(Q_{n-1}^1(x), Q_{n-1}^2(x)\right) \Psi(x)  B^* x^{n-1}dx\\
&=\frac{r_ n y_0}{s_0}  (\alpha_{n-1}, 0)  \begin{pmatrix} 0& 0\\ 0 & \D\frac{1}{y_{-1}}\end{pmatrix}=(0,0).
\end{split}
\end{equation}
For $j=0$ we have, using \eqref{qtilde} and \eqref{esptil}, that
\begin{equation*}
\begin{split}
\int_{-1}^{1} \left(\widetilde Q_n^1(x), \widetilde Q_n^2(x)\right) \widetilde \Psi  (x)dx&= \int_{-1}^{1} \left(S_n^1(x), S_n^2(x)\right) \bm S_0^{-1}(x) \bm S_0(x) \Psi_S(x)  \bm S_0^*(x)dx \\
&= \left[ \int_{-1}^{1} \left(S_n^1(x), S_n^2(x)\right) \Psi_S(x) dx \right] A^* + \left[ \int_{-1}^{1} \left(S_n^1(x), S_n^2(x)\right) x\Psi_S(x) dx \right] B^*.
\end{split}
\end{equation*}
The second term of the sum of the above expression vanishes as a consequence of \eqref{esphat}, \eqref{qenqhat} and \eqref{cond1}. Indeed, for $n\geq2,$ we have
\begin{equation*}
\begin{split}
 &\left[ \int_{-1}^{1} \left(S_n^1(x), S_n^2(x)\right)x\Psi_S(x) dx \right] B^*= 
 \left[\frac{y_0}{s_0} \int_{-1}^{1} \left(S_n^1(x), S_n^2(x)\right) \Psi(x) dx \right] B^*\\
 &\hspace{3cm}= \left[ \frac{s_ n y_0}{s_0}  \int_{-1}^{1} \left(Q_n^1(x), Q_n^2(x)\right)  \Psi(x)dx + \frac{r_n y_0}{s_0}  \int_{-1}^{1} \left(Q_{n-1}^1(x), Q_{n-1}^2(x)\right) \Psi(x)dx\right] B^*=(0,0).
\end{split}
\end{equation*}
For $n=1$ we can use the same argument as in \eqref{Bbss} and get again $(0,0)$. Now, using \eqref{qhatrec1}, we can write
\begin{equation*}
\left(S_n^1(x), S_n^2(x)\right)=x \sum_{j=0}^{n-1} \frac{1}{x_j} \left( \frac{S_n^1(0)}{S_{j+1}^1(0)} Q_j^1(x) , \frac{S_n^2(0)}{S_{j+1}^2(0)} Q_j^2(x)  \right)+ \left(S_n^1(0),S_n^2(0)\right).
\end{equation*} 
Substituting this in the remaining integral we get
\begin{align*}
\int_{-1}^{1} \left(\widetilde Q_n^1(x), \widetilde Q_n^2(x)\right)& \widetilde \Psi  (x)dx= \left[ \int_{-1}^{1} \left(S_n^1(x), S_n^2(x)\right) \Psi_S(x) dx \right] A^*\\
&= \left(S_n^1(0),S_n^2(0)\right)\left[ \begin{pmatrix} 1/s_0 &0 \\ 0 & 1/r_0\end{pmatrix} -\frac{y_0}{s_0}M_{-1}\right] A^* \\
 & \hspace{.5cm} +\frac{y_0}{s_0} \sum_{j=0}^{n-1} \frac{1}{x_j} \int_{-1}^{1} \left( \frac{S_n^1(0)}{S_{j+1}^1(0)} Q_j^1(x) , \frac{S_n^2(0)}{S_{j+1}^2(0)} Q_j^2(x)  \right) \Psi (x)A^*dx+\frac{y_0}{s_0}\left(S_n^1(0),S_n^2(0)\right) M_{-1} A^*\\
&=\left(S_n^1(0),S_n^2(0)\right) \begin{pmatrix} 1/s_0 &0 \\ 0 & 1/r_0\end{pmatrix} A^*+\frac{y_0}{s_0 x_0}\frac{S_n^1(0)}{S_{1}^1(0)} \int_{-1}^1\left(Q_0^1(x),Q_0^2(x)\right)\Psi(x)A^*dx.
\end{align*}
The third step is a consequence of $\D\frac{S_n^1(0)}{S_{j+1}^1(0)}=\frac{S_n^2(0)}{S_{j+1}^2(0)}$ using \eqref{relq0} and the orthogonality properties. Since $\D\int_{-1}^1\left(Q_0^1(x),Q_0^2(x)\right)\Psi(x)dx=(1,0)$, $S_{1}^1(0)=-s_0y_0/x_0$, and $A$ is given by \eqref{q0ext} then we have
\begin{equation*}
\begin{split}
\int_{-1}^{1} \left(\widetilde Q_n^1(x), \widetilde Q_n^2(x)\right)& \widetilde \Psi  (x)dx=\left(S_n^1(0),S_n^2(0)\right) \begin{pmatrix} 1/s_0 &0 \\ 0 & 1/r_0\end{pmatrix} A^*-\frac{1}{s_0^2}\left(S_n^1(0),0\right)A^*\\
&=\left( (1-1/s_0)\frac{S_n^1(0)}{s_0} , \frac{S_n^2(0)}{r_0}  \right) \begin{pmatrix} s_0& -\D\frac{x_{-1}s_0}{y_{-1}}\\ r_0 & -\D\frac{x_{-1}r_0}{y_{-1}}\end{pmatrix}\\
&=\left( (-r_0/s_0)S_n^1(0)+S_n^2(0), -\frac{x_{-1}}{y_{-1}} \left( (-r_0/s_0)S_n^1(0) +S_n^2(0) \right) \right)=(0,0),
\end{split}
\end{equation*}
as a consequence of \eqref{relq0}.

For $n\leq-1$ the proof is similar but now using \eqref{cond2} and \eqref{qhatrec2}. Therefore we have proved \eqref{qusnorms} for $n\neq m$. Observe that this implies in particular that the family of vector-valued polynomials $\left(S_n^1(x),S_n^2(x)\right), n\in\ZZ,$ is also orthogonal (for $n\neq m$) with respect to the weight matrix $\Psi_S(x)$ in \eqref{esphat}. For $n=m$, using this fact and \eqref{qtilde}, \eqref{qen1}, \eqref{qenqhat}, \eqref{esptil} and \eqref{esphat}, we have that
\begin{equation*}
\begin{split}
\int_{-1}^{1} &\left(\widetilde Q_n^1(x), \widetilde Q_n^2(x)\right) \widetilde \Psi(x)   \left(\widetilde Q_n^1(x) , \widetilde Q_n^2(x)\right)^*dx= \int_{-1}^{1}   \left(S_n^1(x),S_n^2(x)\right) \Psi_S(x)  \left(S_n^1(x), S_n^2(x)\right) ^* dx\\
&= \int_{-1}^{1}   \left[ \frac{x}{y_n}\left( Q_n^1(x), Q_n^2(x)\right)-\frac{ x_n}{ y_n} \left(S_{n+1}^1(x), S_{n+1}^2(x)\right) \right] \Psi_S(x)  \left(S_n^1(x), S_n^2(x)\right)^*dx\\
&= \frac{1}{y_n}\int_{-1}^{1}   x \left( Q_n^1(x), Q_n^2(x)\right) \Psi_S(x)  \left(S_n^1(x),S_n^2(x)\right)^*dx - \frac{x_n}{y_n} \int_{-1}^{1}   \left(S_{n+1}^1(x) , S_{n+1}^2(x)\right)\Psi_S(x)  \left(S_n^1(x) , S_n^2(x)\right)^*dx \\
&= \frac{y_0}{y_n s_0}\int_{-1}^{1}    \left( Q_n^1(x) , Q_n^2(x)\right)\Psi(x)\left[  s_n\left( Q_n^1(x) ,  Q_n^2(x)\right)^* + r_n\left( Q_{n-1}^1(x) ,  Q_{n-1}^2(x)\right)^* \right] \\
&=  \frac{ s_n y_0}{y_n s_0 }\int_{-1}^{1}    \left( Q_n^1(x) , Q_n^2(x)\right) \Psi(x) \left( Q_n^1(x) ,  Q_n^2(x)\right)^*dx + \frac{ r_n y_0}{y_n s_0}\int_{-1}^{1}    \left( Q_n^1(x) , Q_n^2(x)\right) \Psi(x) \left( Q_{n-1}^1(x) ,  Q_{n-1}^2(x)\right)^* dx \\
&= \frac{ s_n y_0}{y_n s_0 }\int_{-1}^{1}    \left( Q_n^1(x) , Q_n^2(x)\right)\Psi(x) \left( Q_n^1(x) ,  Q_n^2(x)\right)^*dx =\frac{ s_n y_0}{y_n s_0 } \frac {1}{\pi _n}= \frac{1}{\tilde \pi_n}.
\end{split}
\end{equation*}
The last step follows using \eqref{ULd}, \eqref{coeffDLU} and the definition of $\pi_n$ and $\tilde \pi_n$ in \eqref{potcoeff} and  \eqref{cofpotUL}, respectively.
\end{proof}
\begin{remark}
The spectral matrix $\widetilde\Psi$ associated with the Darboux transformation given in the previous theorem is a \emph{conjugation} by a matrix polynomial of degree 1 (namely $\bm S_0(x)$) of a \emph{Geronimus transformation} of the original spectral matrix $\Psi$. This phenomenon was already present, except for the conjugation, in the case of the Darboux transformation of transition probability matrices on $\ZZ_{\geq0}$ for the UL factorization (see \cite{GdI3}). 
\end{remark}

\begin{remark}
In this paper we have only considered doubly infinite stochastic Jacobi matrices of the form \eqref{QZ}, i.e. all entries are nonnegative and $a_n+b_n+c_n=1,n\in\ZZ$. Certainly the same approach can be used to more general doubly infinite Jacobi matrices as long as we are in the conditions of Corollary 2.2 of \cite{MR}, i.e. to guarantee that the Jacobi matrix is self-adjoint.  The spectral matrix associated with the Darboux transformation will be similar, up to constants, to the one given in Theorem \ref{thmorto}.
\end{remark}

\subsection{Darboux transformation for the LU case}

Consider now the discrete Darboux transformation $\widehat P$ in \eqref{QZLUdar} with probability coefficients $(\hat{a}_n)_{n\in\ZZ}$, $(\hat{b}_n)_{n\in\ZZ}$ and $(\hat{c}_n)_{n\in\ZZ}$ given by \eqref{coeffDUL}. We will see that this case is similar to the one in the previous subsection, so we will only give the important formulas necessary to prove the main results. As in the UL case we need to introduce the auxiliary family of polynomials $T_n^\alpha(x)$ given by the relation $t^\alpha(x)=\widetilde P_Uq^\alpha(x)$, where $q^\alpha(x)=(\cdots,Q_{-1}^\alpha(x),Q_0^\alpha(x),Q_1^\alpha(x),\cdots)^T,$ and $t^\alpha(x)=(\cdots, T_{-1}^\alpha(x), T_0^\alpha(x),T_1^\alpha(x),\cdots)^T, \alpha=1,2$, i.e. 
\begin{equation}\label{qenqhat2}
T_n^\alpha(x)=\tilde y_n Q_n^\alpha (x)+ \tilde x_n Q_{n+1}^\alpha(x),\quad n \in \mathbb{Z}, \quad \alpha=1,2.
\end{equation}
From the LU factorization we also have $\widetilde P_L t^\alpha(x)= xq^\alpha(x)$, that is
\begin{equation}\label{qen2}
xQ_n^\alpha (x)=\tilde r_n T_{n-1}^\alpha(x)+\tilde s_n T_n^\alpha(x), \quad n\in \mathbb{Z}, \quad\alpha=1,2.
\end{equation}
Evaluating \eqref{qen2} at $x=0$ we get recursively
\begin{equation}
\begin{split}\label{q0hat2}
T_{n}^\alpha (0)&=(-1)^{n+1}\frac{\tilde r_0\dots\tilde r_{n}}{\tilde s_0 \dots \tilde s_{n}}T_{-1}^\alpha (0), \quad n\ge 1,\\
T_{-n-1}^\alpha (0)&=(-1)^{n} \frac{\tilde s_{-1}\dots \tilde s_{-n}}{\tilde r_{-1}\dots \tilde r_{-n}} T_{-1}^\alpha(0), \quad n\ge 0,
\end{split}
\end{equation}
where 
\begin{equation*}\label{relacons2}
T_{-1}^\alpha(0)=\begin{cases} \tilde x_{-1}, \quad \text{if}\quad \alpha=1, \\ \tilde y_{-1}, \quad \text{if} \quad  \alpha=2. \end{cases}
\end{equation*}
The equations \eqref{q0hat2} establish a direct relation between the polynomials $(T_n^\alpha)_{n\in\ZZ}, \alpha=1,2,$ given by
\begin{equation} \label{relq02}
\tilde x_{-1} T_n^2(0)=\tilde y_{-1} T_n^1(0), \quad n\in \mathbb{Z}.
\end{equation}
Again, the polynomials $(T_n^\alpha)_{n\in \mathbb{Z}}$ can be written in terms of the polynomials $(Q_n^\alpha)_{n\in \mathbb{Z}}$ as follows
\begin{equation} \label{qhatrec3}
\begin{split}
T_n^\alpha(x)&=T_{n}^\alpha(0) \left[ 1+x\sum_{j=0}^{n} \frac{Q_j^\alpha(x) }{T_{j}^\alpha(0) \tilde s_j} \right], \quad n\ge 1,\\
T_{-n-1}^\alpha(x)&=T_{-n-1}^\alpha(0) \left[ 1+x\sum_{j=0}^{n-1} \frac{Q_{-j-1}^\alpha(x) }{T_{-j-2}^\alpha(0) \tilde r_{-j-1}} \right], \quad n\ge 0.
\end{split}
\end{equation}
This auxiliary family $(T_n^\alpha)_{n\in \mathbb{Z}}$ does not satisfy the same initial conditions as the family $(Q_n^\alpha)_{n\in \mathbb{Z}}$. In fact we have
\begin{equation*}
\begin{split}
T_0^1(x)=\frac{x-\tilde r_0\tilde x_{-1}}{\tilde s_{0}},& \hspace{0.5cm} T_0^2(x)=-\frac{\tilde r_0\tilde y_{-1}}{\tilde s_{0}},\\
T_{-1}^1(x)=\tilde x_{-1},&\hspace{0.5cm}  T_{-1}^2(x)=\tilde y_{-1}.
\end{split}
\end{equation*}
The degrees of the polynomials $(T_n^\alpha)_{n\in \mathbb{Z}}$ are
\begin{equation*}\label{degqss2}
\begin{split}
\deg(T_n^1)&=n+1, \quad  n\ge 0, \quad \deg(T_n^2)=n,\quad n\ge 0,\\
\deg(T_{-n-1}^1)&=n-1, \quad n\ge 0, \quad \deg(T_{-n-1}^2)=n,\quad n\ge 0.
\end{split}
\end{equation*}
Therefore the family of matrix polynomials
\begin{equation*}\label{S002}
\bm T_n(x)=\begin{pmatrix} T_n^1(x)&T_n^2(x) \\ T_{-n-1}^1(x) & T_{-n-1}^2(x)\end{pmatrix}, \quad n\geq0,
\end{equation*}
has degree $n+1$ and singular leading coefficient. Now we will define a new family of polynomials which will turn out to be the associated family of the Darboux transformation $\widehat P$. For $n\geq0$ define 
\begin{equation}\label{qtilde2}
\widehat{\bm Q}_n(x)=\bm T_n(x) \bm T_0^{-1}(x),\quad n\geq0,
\end{equation}
where
\begin{equation}\label{S0002}
\bm T_0(x)=\begin{pmatrix} \D\frac{x-\tilde r_0\tilde x_{-1}}{\tilde s_{0}}& \D-\frac{\tilde r_0\tilde y_{-1}}{\tilde s_{0}} \\ \tilde x_{-1}&\tilde y_{-1}
\end{pmatrix}.
\end{equation}
Following the same representation as in \eqref{2QMM} we can define the functions $(\widehat Q_n^\alpha)_{n\in \mathbb{Z}}, \alpha=1,2$, which turn out to be polynomials, as the following proposition shows.

\begin{proposition}
Let $\widehat{\bm Q}_n(x), n\geq0,$ be the matrix function defined by \eqref{qtilde2}. Then, for $n\geq0$, $\widehat{\bm Q}_n(x)$ is a matrix polynomial of degree exactly $n$ with nonsingular leading coefficient and $\widehat{\bm Q}_0(x)=I_{2\times2}$.
\end{proposition}
\begin{proof}
Now, since
$$
\bm T_{0}^{-1}(x)=\frac{1}{x}\begin{pmatrix} \tilde s_0 & \tilde r_0 \\ -\D\frac{\tilde s_{0}\tilde x_{-1}}{\tilde y_{-1}}& \D\frac{x-\tilde r_0\tilde x_{-1}}{\tilde y_{-1}}\end{pmatrix},
$$
we have from \eqref{qtilde2}
\begin{equation*}\label{qtilss2}
\begin{split}
\widehat Q_n^1(x)&=\frac{\tilde s_{0}}{x\tilde y_{-1}}\left(\tilde y_{-1}T_n^1(x)-\tilde x_{-1}T_n^2(x)\right),\quad n\in\ZZ,\\
\widehat Q_n^2(x)&=\frac{T_n^2(x)}{\tilde y_{-1}}+\frac{\tilde r_{0}}{x\tilde y_{-1}}\left(\tilde y_{-1}T_n^1(x)-\tilde x_{-1}T_n^2(x)\right),\quad n\in\ZZ.
\end{split}
\end{equation*}
From here the proof follows the same lines as in the proof of Proposition \ref{proppol} but now using \eqref{relq02} and \eqref{qhatrec3}.
\end{proof}
Again these new \emph{Darboux polynomials} $(\widehat Q_n^\alpha)_{n\in \mathbb{Z}}, \alpha=1,2,$ satisfy the same initial conditions and degree conditions than the original polynomials $(Q_n^\alpha)_{n\in \mathbb{Z}}, \alpha=1,2,$ as well as a three-term recurrence relation of the form \eqref{TTRRZUL} but with Darboux coefficients $(\hat{a}_n)_{n\in\ZZ}$, $(\hat{b}_n)_{n\in\ZZ}$ and $(\hat{c}_n)_{n\in\ZZ}$ given by \eqref{coeffDUL}.

Let us now prove the analogue of Theorem \ref{thmorto} for the spectral matrix for the polynomials $(\widehat Q_n^\alpha)_{n\in \mathbb{Z}}, \alpha=1,2$. We define first the potential coefficients associated with $\widehat P$ given by
\begin{equation}\label{cofpotUL2}
\hat \pi_0=1, \quad \hat \pi_n=\frac{\hat a_0 \cdots \hat a_{n-1}}{\hat c_1 \cdots \hat c_{n}}, \quad \hat \pi_{-n}=\frac{\hat c_0 \cdots \hat c_{-n+1}}{ \hat a_{-1} \cdots \hat a_{-n}}, \quad n\ge 1.
\end{equation}
\begin{theorem}\label{thmorto2}
Let $\{X_t : t=0, 1, \dots\}$ be the random walk on $\ZZ$ with transition probability matrix $P$ given by \eqref{QZ} and $\{\widehat X_t : t=0, 1, \dots\}$ the Darboux random walk on $\ZZ$ with transition probability matrix $\widehat P$ given by \eqref{QZLUdar}. Assume that $M_{-1}=\D\int_{-1}^{1} \D\frac{\Psi(x)}{x}dx$ is well-defined (entry by entry), where $\Psi(x)$ is the original spectral matrix (see \eqref{2spmt}). Then the polynomials $(\widehat Q_n^\alpha)_{n\in \mathbb{Z}}, \alpha=1,2,$ defined by \eqref{qtilde2} are orthogonal with respect to the following spectral matrix 
\begin{equation}\label{esptil2}
\widehat \Psi (x)=\bm T_0(x) \Psi_T (x) \bm T_0^*(x),
\end{equation}
where $\bm T_0(x)$ is defined by \eqref{S0002} and
\begin{equation}\label{esphat2}
\Psi_T(x)=\frac{\tilde s_0}{\tilde y_0}\frac{\Psi(x)}{x}+ \left[\frac{\hat{a}_{-1}}{\hat{c}_0}\begin{pmatrix}1/\tilde x_{-1}&0 \\ 0 &1/\tilde y_{-1} \end{pmatrix}- \frac{\tilde s_0}{\tilde y_0}M_{-1}\right] \delta_0(x).
\end{equation}
Moreover, we have
\begin{equation*}\label{qusnorms2}
\int_{-1}^{1} \widehat{\bm Q}_n (x) \widehat \Psi (x) \widehat{\bm Q}_m^*(x) dx=\begin{pmatrix} 1/\hat \pi_n &0\\ 0& 1/\hat \pi_{-n-1}\end{pmatrix} \delta_{n,m},
\end{equation*}
where $(\hat \pi_n)_{n\in \mathbb{Z}}$ are the potential coefficients defined by \eqref{cofpotUL2}.
\end{theorem}
\begin{proof}
The proof follows the same lines as the proof of Theorem \ref{thmorto}. For $n\geq1$ and $j=1,\ldots,n-1$ we have, using \eqref{esptil2}, \eqref{qtilde2}, \eqref{esphat2} and \eqref{qenqhat2}, that
\begin{equation*}
\begin{split}
\int_{-1}^{1} \left(\widehat Q_n^1(x), \widehat Q_n^2(x)\right)& \widehat \Psi  (x)x^jdx= \frac{\tilde y_n \tilde s_0}{\tilde y_0}\left[\int_{-1}^{1} \left(Q_n^1(x), Q_n^2(x)\right) \Psi(x)  A^* x^{j-1}dx+\int_{-1}^{1} \left(Q_n^1(x), Q_n^2(x)\right) \Psi(x)  B^* x^{j}dx\right]\\
&\quad+  \frac{\tilde x_n \tilde s_0}{\tilde y_0}\left[\int_{-1}^{1} \left(Q_{n+1}^1(x), Q_{n+1}^2(x)\right) \Psi(x)  A^* x^{j-1}dx+\int_{-1}^{1} \left(Q_{n+1}^1(x), Q_{n+1}^2(x)\right) \Psi(x)  B^* x^{j}dx\right].
\end{split}
\end{equation*}
where $A$ and $B$ (from $\bm T_0(x)=A+ xB$),  are given now by 
\begin{equation} \label{q0ext2}
A=\begin{pmatrix} -\D\frac{\tilde r_0\tilde x_{-1}}{\tilde s_{0}}& -\D\frac{\tilde r_0\tilde y_{-1}}{\tilde s_{0}}\\ \tilde x_{-1} &\tilde y_{-1} \end{pmatrix},\quad B=\begin{pmatrix} \D\frac{1}{\tilde s_{0}}& 0\\ 0 & 0\end{pmatrix},
\end{equation}
Using \eqref{cond1} we have that the first term of the sum vanishes for $j=1,\ldots,n$, the second term vanishes for $j=0,\ldots,n-1$, the third term vanishes for $j=1,\ldots,n+1$ and the fourth term vanishes for $j=0,\ldots,n$. Therefore the above expression vanishes for $j=1,\ldots,n-1$. 

For $j=0$ we have, using \eqref{qtilde2} and \eqref{esptil2}, that
\begin{equation*}
\begin{split}
\int_{-1}^{1} \left(\widehat Q_n^1(x), \widehat Q_n^2(x)\right) \widehat \Psi  (x)dx&= \int_{-1}^{1} \left(T_n^1(x), T_n^2(x)\right) \bm T_0^{-1}(x) \bm T_0(x) \Psi_T(x)  \bm T_0^*(x)dx \\
&= \left[ \int_{-1}^{1} \left(T_n^1(x), T_n^2(x)\right) \Psi_T(x) dx \right] A^* + \left[ \int_{-1}^{1} \left(T_n^1(x), T_n^2(x)\right) x\Psi_T(x) dx \right] B^*.
\end{split}
\end{equation*}
As before, the second term of the sum of the above expression vanishes as a consequence of \eqref{esphat2}, \eqref{qenqhat2} and \eqref{cond1}. Now, using \eqref{qhatrec3}, we can write
\begin{equation*}
\left(T_n^1(x), T_n^2(x)\right)=x \sum_{j=0}^{n} \frac{1}{\tilde s_j} \left( \frac{T_n^1(0)}{T_{j}^1(0)} Q_j^1(x) , \frac{T_n^2(0)}{T_{j}^2(0)} Q_j^2(x)  \right)+ \left(T_n^1(0),T_n^2(0)\right).
\end{equation*} 
Substituting this in the remaining integral we get
\begin{align*}
\int_{-1}^{1} \left(\widehat Q_n^1(x), \widehat Q_n^2(x)\right)& \widehat \Psi  (x)dx= \left[ \int_{-1}^{1} \left(T_n^1(x), T_n^2(x)\right) \Psi_T(x) dx \right] A^*\\
&= \left(T_n^1(0),T_n^2(0)\right)\left[\frac{\hat{a}_{-1}}{\hat{c}_0}\begin{pmatrix} 1/\tilde x_{-1} &0 \\ 0 & 1/\tilde y_{-1}\end{pmatrix} -\frac{\tilde s_0}{\tilde y_0}M_{-1}\right] A^* \\
 & \hspace{.5cm} +\frac{\tilde s_0}{\tilde y_0} \sum_{j=0}^{n} \frac{1}{\tilde s_j} \int_{-1}^{1} \left( \frac{T_n^1(0)}{T_{j}^1(0)} Q_j^1(x) , \frac{T_n^2(0)}{T_{j}^2(0)} Q_j^2(x)  \right) \Psi (x)A^*dx+\frac{\tilde s_0}{\tilde y_0} \left(T_n^1(0),T_n^2(0)\right) M_{-1} A^*\\
&=\frac{\hat{a}_{-1}}{\hat{c}_0}\left(T_n^1(0),T_n^2(0)\right) \begin{pmatrix} 1/\tilde x_{-1} &0 \\ 0 & 1/\tilde y_{-1}\end{pmatrix} A^*+\frac{1}{\tilde y_0} \frac{T_n^1(0)}{T_{0}^1(0)}\int_{-1}^1\left(Q_0^1(x),Q_0^2(x)\right)\Psi(x)A^*dx.
\end{align*}
The third step is a consequence of $\D\frac{T_n^1(0)}{T_{j}^1(0)}=\frac{T_n^2(0)}{T_{j}^2(0)}$ using \eqref{relq02} and the orthogonality properties. Since $\D\int_{-1}^1\left(Q_0^1(x),Q_0^2(x)\right)\Psi(x)dx=(1,0)$, $T_{0}^1(0)=-\tilde r_0\tilde x_{-1}/\tilde s_0$, $\hat{a}_{-1}/\hat{c}_0=\tilde x_{-1}\tilde s_{0}/\tilde y_0\tilde r_0$ and $A$ is given by \eqref{q0ext2} then we have
\begin{equation*}
\begin{split}
\int_{-1}^{1} \left(\widehat Q_n^1(x), \widehat Q_n^2(x)\right)& \widehat \Psi  (x)dx=\frac{\tilde x_{-1}\tilde s_{0}}{\tilde y_0\tilde r_0}\left(T_n^1(0),T_n^2(0)\right) \begin{pmatrix} 1/\tilde x_{-1} &0 \\ 0 & 1/\tilde y_{-1}\end{pmatrix} A^*-\frac{\tilde s_0}{\tilde y_0\tilde r_0\tilde x_{-1}}\left(T_n^1(0),0\right)A^*\\
&= \left(\frac{\tilde s_0}{\tilde y_0\tilde r_0}(1-1/\tilde x_{-1})T_n^1(0), \frac{\tilde x_{-1}\tilde s_0}{\tilde r_0\tilde y_0\tilde y_{-1}} T_n^2(0) \right) \begin{pmatrix} -\D\frac{\tilde r_0\tilde x_{-1}}{\tilde s_{0}}&\tilde x_{-1}  \\ -\D\frac{\tilde r_0\tilde y_{-1}}{\tilde s_{0}}&\tilde y_{-1} \end{pmatrix}\\
&=\left( \frac{1}{\tilde y_0}(\tilde y_{-1}T_n^1(0)-\tilde x_{-1}T_n^2(0)), -\frac{\tilde s_0}{\tilde y_0\tilde r_0}(\tilde y_{-1}T_n^1(0)-\tilde x_{-1}T_n^2(0)) \right)=(0,0),
\end{split}
\end{equation*}
as a consequence of \eqref{relq02}. For $n\leq-1$ the proof is similar but now using \eqref{cond2} and \eqref{qhatrec3}. Finally, using \eqref{qtilde2}, \eqref{qen2}, \eqref{qenqhat2}, \eqref{esptil2} and \eqref{esphat2}, we have that
\begin{equation*}
\begin{split}
\int_{-1}^{1} &\left(\widehat Q_n^1(x), \widehat Q_n^2(x)\right) \widehat \Psi(x)   \left(\widehat Q_n^1(x) , \widehat Q_n^2(x)\right)^*dx= \int_{-1}^{1}   \left(T_n^1(x),T_n^2(x)\right) \Psi_T(x)  \left(T_n^1(x), T_n^2(x)\right) ^* dx\\
&= \int_{-1}^{1}   \left[ \frac{x}{\tilde s_n}\left( Q_n^1(x), Q_n^2(x)\right)-\frac{\tilde r_n}{ \tilde s_n} \left(T_{n-1}^1(x), T_{n-1}^2(x)\right) \right] \Psi_T(x)  \left(T_n^1(x), T_n^2(x)\right)^*dx\\
&= \frac{1}{\tilde s_n}\int_{-1}^{1}   x \left( Q_n^1(x), Q_n^2(x)\right) \Psi_T(x)  \left(T_n^1(x),T_n^2(x)\right)^*dx\\
&= \frac{\tilde s_0}{\tilde s_n \tilde y_0}\int_{-1}^{1}    \left( Q_n^1(x) , Q_n^2(x)\right)\Psi(x)\left[  \tilde y_n\left( Q_n^1(x) ,  Q_n^2(x)\right)^* + \tilde x_n\left( Q_{n+1}^1(x) ,  Q_{n+1}^2(x)\right)^* \right] \\
&=  \frac{\tilde y_n\tilde s_0}{\tilde s_n \tilde y_0}\int_{-1}^{1}    \left( Q_n^1(x) , Q_n^2(x)\right) \Psi(x) \left( Q_n^1(x) ,  Q_n^2(x)\right)^*dx=\frac{\tilde y_n\tilde s_0}{\tilde s_n \tilde y_0}\frac {1}{\pi _n}= \frac{1}{\hat \pi_n}.
\end{split}
\end{equation*}
The last step follows using \eqref{ULd}, \eqref{coeffDUL} and the definition of $\pi_n$ and $\hat \pi_n$ in \eqref{potcoeff} and  \eqref{cofpotUL2}, respectively.

\end{proof}

\begin{remark}
The spectral matrix $\widehat\Psi$ associated with the Darboux transformation given in the previous theorem is \emph{again} a conjugation by a matrix polynomial of degree 1 (namely $\bm T_0(x)$) of a \emph{Geronimus transformation} of the original spectral matrix $\Psi$. This phenomenon is different from a Darboux transformation of a transition probability matrix on $\ZZ_{\geq0}$ for the LU factorization (see \cite{GdI3}), where the associated spectral measure is given by a \emph{Christoffel transformation}, i.e. multiplying the original measure by the polynomial $x$.
\end{remark}

\section{Examples}\label{sec4}

\subsection{Random walk on $\ZZ$ with constant transition probabilities} 

Consider $P$ as in \eqref{QZ} with coefficients
$$
a_n=a,\quad b_n=b,\quad c_n=c,\quad n\in\mathbb{Z},\quad a+b+c=1,\quad a,c>0, b\geq0.
$$
For the UL factorization, the continued fractions \eqref{cfUL} can be computed explicitly. Indeed, using Proposition 4.1 of \cite{GdI3} we have that
\begin{equation*}\label{FFF1}
H=\frac{1}{2}\left(1+c-a+\sqrt{(1+c-a)^2-4c}\right),
\end{equation*}
as long as $a\leq(1-\sqrt{c})^2$ (to ensure convergence). On the other hand, we have $H'=c/H$. Therefore we have
\begin{equation*}\label{FFF2}
H'=\frac{1}{2}\left(1+c-a-\sqrt{(1+c-a)^2-4c}\right).
\end{equation*}
It is also possible to see, under the conditions on the parameters, that $0\leq H'\leq H\leq1$. Therefore, according to Theorem \ref{thmfracUL}, as long as $H'\leq y_0\leq H$ we always have a stochastic UL factorization where both factors are stochastic matrices. All formulas simplify considerably when $y_0=H$ (or $y_0=H'$). Indeed, in this case we have
\begin{equation}\label{coeff}
\begin{split}
y_n&=H,\quad x_n=1-H,\\
s_n&=1-\frac{c}{H},\quad r_n=\frac{c}{H},
\end{split}
\end{equation}
while if $y_0=H'$ we have the same formulas but replacing $H'$ by $H$. It is also remarkable that in these cases the coefficients of the Darboux transformation \eqref{coeffDLU} \emph{remain invariant}, i.e. the random walk $\widetilde P$ is exactly the same as the original random walk $P$. This phenomenon is not possible for Darboux transformations for random walks on $\ZZ_{\geq0}$.

As for the LU factorization, following Theorem \ref{thmfracLU}, we have a stochastic LU factorization if and only if we choose the free parameter $\tilde r_0$ in the range $H'\leq \tilde r_0\leq H$. In this case we also have that if we choose $\tilde r_0$ either $H$ or $H'$ then the coefficients of the Darboux transformation \eqref{coeffDUL} \emph{remain invariant}, i.e. the random walk $\widehat P$ is exactly the same as the original random walk $P$.

\medskip

The spectral matrix associated with this example appeared for the first time in the last section of \cite{KMc6} (for the case of $b=0$, i.e. the symmetric random walk) along with a method to compute the spectral matrix using Stieltjes transforms and the spectral measures associated with the positive and negative states of the original random walk. A combination of this method and Proposition 4.2 of \cite{GdI3} shows that the spectral matrix of the original random walk is given by only an absolutely continuous part, i.e.
\begin{equation}\label{WW1}
\Psi(x)=\frac{1}{\pi\sqrt{(x-\sigma_-)(\sigma_+-x)}}\begin{pmatrix} 1 & \D\frac{x-b}{2c}\\\D\frac{x-b}{2c}& a/c\end{pmatrix},\quad x\in[\sigma_-,\sigma_+],\quad \sigma_{\pm}=1-\left(\sqrt{a}\mp\sqrt{c}\right)^2.
\end{equation}
Therefore we get the Karlin-McGregor formula \eqref{KmcG1} for the $n$-step transition probabilities of the random walk  $P$. Also, using  Remark \ref{remrec}, we have that the random walk is always transient except for the case $a=c$. The random walk is never positive recurrent since the spectral matrix \eqref{WW1} does not have a jump at the point 1, so for the case $a=c$ the random walk is null recurrent.

A straightforward computation shows that the moment $M_{-1}$ of $\Psi$ is given by
\begin{equation}\label{Mm11}
M_{-1}=\begin{pmatrix} \D\frac{1}{\sqrt{\sigma_-\sigma_+}} &\D\frac{1}{2c}\left(1-\D\frac{b}{\sqrt{\sigma_-\sigma_+}}\right)\\\D\frac{1}{2c}\left(1-\D\frac{b}{\sqrt{\sigma_-\sigma_+}}\right)& \D\frac{a}{c\sqrt{\sigma_-\sigma_+}}\end{pmatrix}.
\end{equation}
In order for $M_{-1}$ to be well-defined we need to assume that $\sigma_->0$, i.e. $\sqrt{a}+\sqrt{c}<1$, or, in other words $a<(1-\sqrt{c})^2$, which is the condition for convergence of the continued fractions $H$ and $H'$. With this information we can compute the spectral matrices associated with the Darboux transformation $\widetilde P$ in \eqref{QZULdar} (for the UL factorization) and $\widehat P$ in \eqref{QZLUdar} (for the LU factorization), both depending on one free parameter.

For the UL case we have, using Theorem \ref{thmorto}, that
$$
\widetilde \Psi(x)=\bm S_0(x) \Psi_S (x) \bm S_0^*(x),
$$
where
\begin{equation*}
\bm S_0(x)=\begin{pmatrix} s_0& r_0 \\ -\D\frac{x_{-1}s_0}{y_{-1}}& \D\frac{x-x_{-1}r_0}{y_{-1}}
\end{pmatrix}=\begin{pmatrix} \D\frac{y_0-c}{y_0}& \D\frac{c}{y_0} \\ -\D\frac{a(y_0-c)}{y_0(1-a)-c}& \D\frac{x(y_0-c)-ac}{y_0(1-a)-c}
\end{pmatrix},
\end{equation*}
and
\begin{equation*}
\Psi_S(x)=\frac{y_0}{y_0-c}\left(y_0\frac{\Psi(x)}{x}+ \left[ \begin{pmatrix}1&0 \\ 0 &\D\frac{y_0-c}{c} \end{pmatrix}-y_0M_{-1}\right] \delta_0(x)\right),
\end{equation*}
where $\Psi$ and $M_{-1}$ are defined by \eqref{WW1} and \eqref{Mm11}, respectively. Observe that the only free parameter is $y_0$. A straightforward computation gives that
\begin{equation*}
\widetilde\Psi(x)=\frac{1}{\pi x\sqrt{(x-\sigma_-)(\sigma_+-x)}}\left[\widetilde A+\widetilde Bx+\widetilde Cx^2\right]+\widetilde{\bm M}\delta_0(x),
\end{equation*}
where
\begin{align*}
\widetilde A&=\frac{(H'-y_0)(H-y_0)}{s_0y_0}\begin{pmatrix}1&-x_{-1}/y_{-1} \\ -x_{-1}/y_{-1}&(x_{-1}/y_{-1})^2 \end{pmatrix},\\
\widetilde B&=\begin{pmatrix}1&-\D\frac{by_0}{2cy_{-1}} \\ -\D\frac{by_0}{2cy_{-1}} &\D\frac{(y_0b-c(1-c))x_{-1}^2}{acy_{-1}^2} \end{pmatrix},\quad \widetilde C=\frac{y_0}{2cy_{-1}}\begin{pmatrix}0&1\\1&0\end{pmatrix},\\
\widetilde{\bm M}&=\frac{(y_0-H')(H-y_0)}{s_0y_0\sqrt{\sigma_-\sigma_+}}\begin{pmatrix}1&-x_{-1}/y_{-1} \\ -x_{-1}/y_{-1}&(x_{-1}/y_{-1})^2 \end{pmatrix}.
\end{align*}
From here we clearly see that if we choose $y_0$ in the range $H'\leq y_0\leq H$, then $\widetilde{\bm M}$ is a positive semidefinite matrix, so $\widetilde\Psi$ is a proper \emph{weight matrix}. Another interesting case, as we mentioned earlier, is when we choose either $y_0=H$ or $y_0=H'$. In these cases we have $\widetilde A=0_{2\times2}, \widetilde{\bm M}=0_{2\times2}$ and following \eqref{coeff} we get
$$
\widetilde B=\begin{pmatrix}1&-\D\frac{b}{2c} \\ -\D\frac{b}{2c}&a/c\end{pmatrix},\quad \widetilde C=\frac{1}{2c}\begin{pmatrix}0&1\\1&0\end{pmatrix}.
$$
Therefore we recover the original weight matrix \eqref{WW1}, as we predicted before. From the spectral matrix $\widetilde\Psi$ we get the Karlin-McGregor formula \eqref{KmcG1} for the $n$-step transition probabilities of the random walk $\widetilde P$. The recurrence of the Darboux random walk is not affected by the transformation.

For the LU case we have, using Theorem \ref{thmorto2} and after some computations, that
\begin{equation*}
\widehat\Psi(x)=\frac{1}{\pi x\sqrt{(x-\sigma_-)(\sigma_+-x)}}\left[\widehat A+\widehat Bx+\widehat Cx^2\right]+\widehat{\bm M}\delta_0(x),
\end{equation*}
where
\begin{align*}
\widehat A&=\frac{\tilde r_0(H'-\tilde r_0)(H-\tilde r_0)}{\tilde x_{-1}\tilde s_0^2}\begin{pmatrix}1&-\tilde s_{0}/\tilde r_{0} \\ -\tilde s_{0}/\tilde r_{0}&(\tilde s_{0}/\tilde r_{0})^2 \end{pmatrix},\\
\widehat B&=\frac{\tilde r_0\tilde y_0}{\tilde x_{-1}\tilde s_0}\begin{pmatrix}1&-\D\frac{b}{2\tilde y_{0}\tilde r_0} \\ -\D\frac{b}{2\tilde y_{0}\tilde r_0} &\D\frac{\tilde s_0 \tilde x_{-1}}{\tilde y_0\tilde r_0} \end{pmatrix},\quad \widehat C=\frac{1}{2\tilde s_0\tilde x_{-1}}\begin{pmatrix}0&1\\1&0\end{pmatrix},\\
\widehat{\bm M}&=\frac{\tilde r_0(\tilde r_0-H')(H-\tilde r_0)}{\tilde s_0^2\tilde x_{-1}\sqrt{\sigma_-\sigma_+}}\begin{pmatrix}1&-\tilde s_{0}/\tilde r_{0} \\ -\tilde s_{0}/\tilde r_{0}&(\tilde s_{0}/\tilde r_{0})^2 \end{pmatrix}.
\end{align*}
Again we clearly see that if we choose $\tilde r_0$ in the range $H'\leq \tilde r_0\leq H$, then $\widehat{\bm M}$ is a positive semidefinite matrix, so $\widehat\Psi$ is a proper \emph{weight matrix}. If we choose either $\tilde r_0=H$ or $\tilde r_0=H'$, then we recover the original weight matrix \eqref{WW1}, as we predicted before. Finally, from the spectral matrix $\widehat\Psi$ we get the Karlin-McGregor formula \eqref{KmcG1} for the $n$-step transition probabilities of the random walk $\widehat P$ and the recurrence of the Darboux random walk is not affected by the transformation.

\subsection{Random walk on $\ZZ$ with constant transition probabilities and an attractive or repulsive force}

Consider $P$ as in \eqref{QZ} with
$$
a_n=a,\quad c_{n}=c,\quad n\geq0,\quad a_{-n}=c,\quad c_{-n}=a,\quad n\geq1,\quad  b_n=b,\quad n\in\mathbb{Z},
$$
and, as before, $a+b+c=1, a,c>0, b\geq0$. Observe that the probabilities $a$ and $c$ are interchanged for nonnegative and negative states of the random walk. Therefore if $a<c$ we have a random walk where the origin is an attractive state. On the contrary, if $a>c$, then the origin is a repulsive state.

Again, the continued fractions \eqref{cfUL} can be computed explicitly. $H$ is the same as before, i.e.
\begin{equation}\label{FFF1}
H=\frac{1}{2}\left(1+c-a+\sqrt{(1+c-a)^2-4c}\right),
\end{equation}
as long as $a\leq(1-\sqrt{c})^2$ (to ensure convergence). On the other hand, we have
$$
H'=\frac{c}{1-\D\frac{c}{H}}.
$$
Rationalizing we get
\begin{equation}\label{FFF2}
H'=\frac{c}{2a}\left(1+a-c-\sqrt{(1+c-a)^2-4c}\right).
\end{equation}
It is easy to see that $H'>0$ if and only if $a>0$. On the other hand, now it is not true that $H'\leq H$ for all values of the parameters $a$ and $c$ such that $a\leq(1-\sqrt{c})^2$. In fact, as $a$ gets closer to 0, we have that there are some values of $c$ such that $H'>H$. A closer look to the inequality $H'\leq H$ shows that $a$ must be in the following range
\begin{equation}\label{rangg}
\begin{cases}
0<a\leq(1-\sqrt{c})^2,&\mbox{if}\quad 0<c\leq1/4,\\
0<a\leq\D\frac{1-2c}{2},&\mbox{if}\quad 1/4\leq c<1.
\end{cases}
\end{equation}
Now, if we choose $y_0=H$, we observe that the positive states of the original random walk remain invariant under the Darboux transformation \eqref{coeffDLU}, while if $y_0=H'$, then the negative states of the original random walk remain invariant, but the rest of coefficients are difficult to compute.

As for the LU factorization, following Theorem \ref{thmfracLU}, we have a stochastic LU factorization if and only if we choose the free parameter $\tilde r_0$ in the range $H'\leq \tilde r_0\leq H$. As before, in order for $H'\leq H$, $a$ must be in the range \eqref{rangg}. Similar behavior holds if we assume either $\tilde r_0=H$ or $\tilde r_0=H'$ for the random walk 
 $\widehat P$.

\medskip

The spectral matrix associated with this example appeared for the first time in Section 6 of \cite{G1} (for the case of $b=0$, i.e. the symmetric random walk). As before, we can compute the spectral matrix for the original random walk $P$. Now in this case the spectral matrix is given by an absolutely continuous and a discrete part, which we write as $\Psi(x)=\Psi_c(x)+\Psi_d(x)$. The absolutely continuous part is given by
$$
\Psi_c(x)=\frac{(a+c)\sqrt{(x-\sigma_-)(\sigma_+-x)}}{2\pi c(1-x)(x-2b+1)}\begin{pmatrix} 1 & \D\frac{x-b}{a+c}\\\D\frac{x-b}{a+c}& 1\end{pmatrix},\quad x\in[\sigma_-,\sigma_+],
$$
where $\sigma_{\pm}$ are defined in \eqref{WW1}. The discrete part is given by
$$
\Psi_d(x)=\frac{c-a}{2c}\left[\begin{pmatrix}1&-1\\-1&1\end{pmatrix}\delta_{2b-1}(x)+\begin{pmatrix}1&1\\1&1\end{pmatrix}\delta_{1}(x)\right]\chi_{\{c>a\}},
$$
where $\chi_{A}$ is the indicator function. Therefore we get the Karlin-McGregor formula \eqref{KmcG1} for the $n$-step transition probabilities of the random walk  $P$. Also, using  Remark \ref{remrec}, the random walk is always recurrent and positive recurrent if $c>a$ since for that case the spectral matrix has a jump at the point 1.

Now the computation of the moment $M_{-1}$ of $\Psi$ is more complicated since we have an absolutely continuous and a discrete part. Also one of the Dirac deltas of the discrete part can be located at $x=0$ if $b=1/2$, so that $M_{-1}$ may have a different expression in that case. Nevertheless, after some computations, we obtain
\begin{equation*}\label{Mm12}
M_{-1}=\begin{pmatrix} \mu_{-1}&\D\frac{\gamma-b\mu_{-1}}{a+c}\\\D\frac{\gamma-b\mu_{-1}}{a+c}& \mu_{-1}\end{pmatrix}+\frac{c-a}{c(2b-1)}\begin{pmatrix}b&-(a+c)\\-(a+c)&b\end{pmatrix}\chi_{\{c>a\}},
\end{equation*}
where
$$
\mu_{-1}=\frac{1}{2c(2b-1)}\left((a+c)\sqrt{\sigma_-\sigma_+}-b|a-c|\right),\quad \gamma=\begin{cases} 1,&\mbox{if}\quad c\leq a,\\a/c,&\mbox{if}\quad c> a.\end{cases}
$$
In order for $M_{-1}$ to be well-defined we need to assume that $\sigma_->0$, i.e. $\sqrt{a}+\sqrt{c}<1$, or, in other words $a<(1-\sqrt{c})^2$, which is the condition for convergence of the continued fractions $H$ and $H'$. For the case of $b=1/2$ ($c=1/2-a$) we obtain
\begin{equation}\label{Mm122}
M_{-1}=\begin{pmatrix} \D\frac{4a}{|4a-1|}& 2\left(\gamma-\D\frac{2a}{|4a-1|}\right)\\2\left(\gamma-\D\frac{2a}{|4a-1|}\right)&\D\frac{4a}{|4a-1|}\end{pmatrix}+\frac{1-4a}{1-2a}\begin{pmatrix}1&1\\1&1\end{pmatrix}\chi_{\{a<1/4\}}.
\end{equation}

For the UL case we have, using Theorem \ref{thmorto} and after some computations, that
\begin{equation*}
\widetilde\Psi(x)=\frac{\sqrt{(x-\sigma_-)(\sigma_+-x)}}{2\pi cx(1-x)(x-2b+1)}\left[\widetilde A+\widetilde Bx+\widetilde Cx^2\right]+\widetilde{\bm M}_0\delta_0(x)+\widetilde{\bm M}_{2b-1}\delta_{2b-1}(x)+\widetilde{\bm M}_1\delta_1(x),
\end{equation*}
where
\begin{align*}
\widetilde A&=\frac{(a+c)(\alpha_+-y_0)(\alpha_--y_0)}{s_0y_0}\begin{pmatrix}1&-x_{-1}/y_{-1} \\ -x_{-1}/y_{-1}&(x_{-1}/y_{-1})^2 \end{pmatrix},\quad \alpha_{\pm}=\frac{c}{a+c}\left(1\pm\sqrt{1-2a-2c}\right),\\
\widetilde B&=\begin{pmatrix}2c&\D\frac{(c(1-2c)-by_0)x_{-1}}{cy_{-1}} \\ \D\frac{(c(1-2c)-by_0)x_{-1}}{cy_{-1}} &\D\frac{2(by_0-c(1-c))x_{-1}^2}{cy_{-1}^2}\end{pmatrix},\quad \widetilde C=\frac{y_0s_0x_{-1}}{cy_{-1}}\begin{pmatrix}0&1\\1&\D\frac{(a-c)x_{-1}}{cy_{-1}}
\end{pmatrix},\\
\widetilde{\bm M}_0&=\frac{a(\bar{H'}-\bar{H})(y_0-H')(H-y_0)}{c(2b-1)s_0y_0}\begin{pmatrix}1&-x_{-1}/y_{-1} \\ -x_{-1}/y_{-1}&(x_{-1}/y_{-1})^2 \end{pmatrix},\\
\widetilde{\bm M}_{2b-1}&=\frac{c-a}{2c(2b-1)}\begin{pmatrix}(s_0-r_0)^2&-(s_0-r_0)(s_{-1}-r_{-1})\\ -(s_0-r_0)(s_{-1}-r_{-1})&(s_{-1}-r_{-1})^2 \end{pmatrix}\chi_{\{c>a\}},\quad \widetilde{\bm M}_1=\frac{c-a}{2c}\begin{pmatrix}1&1\\1&1\end{pmatrix}\chi_{\{c>a\}}.
\end{align*}
Here $\bar{H}$ and $\bar{H'}$ in $\widetilde{\bm M}_0$ are the radical conjugates of $H$ and $H'$, respectively, and we are implicitly assuming that $b\neq 1/2$. If $b=1/2$ then the Geronimus transformation is not well-defined for the Dirac delta at $x=0$. However, it is possible to define the spectral matrix in terms of the \emph{derivative} of the Dirac delta at $x=0$. Indeed, the spectral matrix for the Darboux transformation is given in this case by 
\begin{equation}\label{spmat12}
\widetilde\Psi(x)=\frac{\sqrt{(x-\sigma_-)(\sigma_+-x)}}{2\pi cx^2(1-x)}\left[\widetilde A+\widetilde Bx+\widetilde Cx^2\right]+\widetilde{\bm M}_0\delta_0(x)-\widetilde{\bm M}_{0}'\delta_{0}'(x)+\widetilde{\bm M}_1\delta_1(x),
\end{equation}
where $\widetilde A,\widetilde B,\widetilde C$ and $\widetilde{\bm M}_1$ are the same as before writing $b=1/2$ and $c=1/2-a$ and
\begin{align*}
\widetilde{\bm M}_0'&=\lim_{b\to1/2}(2b-1)\widetilde{\bm M}_{2b-1},\\
\widetilde{\bm M}_0&=\eta\frac{(y_0-H')(H-y_0)}{s_0y_0}\begin{pmatrix}1&-x_{-1}/y_{-1} \\ -x_{-1}/y_{-1}&(x_{-1}/y_{-1})^2 \end{pmatrix},\quad\eta=\begin{cases} \D\frac{1}{2(1-2a)(1-4a)},&\mbox{if}\quad a<1/4,\\\D\frac{a}{4a-1},&\mbox{if}\quad a>1/4.\end{cases}
\end{align*}
If $a=1/4$ then the moment \eqref{Mm122} is not well-defined. Observe that in this case we are in the situation of the previous example. From the spectral matrix $\widetilde\Psi$ we get the Karlin-McGregor formula \eqref{KmcG1} for the $n$-step transition probabilities of the random walk $\widetilde P$. The recurrence of the Darboux random walk is not affected by the transformation.

\medskip

For the LU case we have, using Theorem \ref{thmorto2} and after some computations, that

\begin{equation*}
\widehat \Psi(x)=\frac{ \sqrt{(x-\sigma_-)(\sigma_+-x)}}{2\pi cx(1-x)(x-2b+1)}\left[\widehat A+\widehat Bx+\widehat Cx^2\right]+\widehat {\bm M}_0\delta_0(x)+\widehat {\bm M}_{2b-1}\delta_{2b-1}(x)+\widehat {\bm M}_1\delta_1(x),
\end{equation*}
where
\begin{align*}
\widehat A&=\frac{(a+c)(\beta_+-\tilde s_0)(\beta_--\tilde s_0)}{\tilde s_0\tilde y_0}\begin{pmatrix}1&-\tilde s_0/\tilde r_0 \\ -\tilde s_0/ \tilde r_0 &(\tilde s_0/\tilde r_0)^2 \end{pmatrix},\quad \beta_{\pm}=\frac{a+c\sqrt{2b-1}}{a+c},\\
\widehat B&=\frac{1}{\tilde y_0\tilde r_0}\begin{pmatrix} \D\frac{2\tilde r_0}{\tilde s_0}(c(1-c)-p\tilde r_0) & (a-c)\tilde r_0-c(1-2c) \\ (a-c)\tilde r_0-c(1-2c)  & 2c\tilde s_0\tilde x_{-1}\end{pmatrix},\quad \widehat C=\frac{\tilde y_{-1}}{\tilde y_0}\begin{pmatrix} \D\frac{a-c}{\tilde y_{-1}} &1\\1&0
\end{pmatrix},\\
\widetilde{\bm M}_0&=\frac{a(\bar{H'}-\bar{H})(\tilde r_0-H')(H-\tilde r_0)}{c(2b-1)\tilde s_0\tilde y_0}\begin{pmatrix}1&-\tilde s_{0}/\tilde r_0 \\ -\tilde s_{0}/\tilde r_0&(\tilde s_{0}/\tilde r_0)^2 \end{pmatrix},\quad \widehat{\bm M}_1=\frac{c-a}{2c}\begin{pmatrix}1&1\\1&1\end{pmatrix}\chi_{\{c>a\}},\\
\widehat {\bm M}_{2b-1}&=\frac{c-a}{2c(2b-1)}\begin{pmatrix}(\tilde x_0-\tilde y_0)^2&-(\tilde x_0-\tilde y_0)(\tilde x_{-1}-\tilde y_{-1})\\ -(\tilde x_0-\tilde y_0)(\tilde x_{-1}-\tilde y_{-1})&(\tilde x_{-1}-\tilde y_{-1})^2 \end{pmatrix}\chi_{\{c>a\}}.
\end{align*}
Similar results hold for the case $b=1/2$. From the spectral matrix $\widehat\Psi$ we get the Karlin-McGregor formula \eqref{KmcG1} for the $n$-step transition probabilities of the random walk $\widehat P$ and the recurrence of the Darboux random walk is not affected by the transformation.

Observe that if we assume $a=c$ then we recover the previous example and the Darboux transformation is invariant if we choose the free parameter $y_0=H=\frac{1}{2}(1+\sqrt{1-4a})$ or $y_0=H'=\frac{1}{2}(1-\sqrt{1-4a})$ (same for the LU factorization). We  have not found any other choice of the free parameter $y_0$ (or $\tilde r_0$) such that the Darboux transformation is invariant. Nevertheless there are some values of the parameters where we can guarantee that the Darboux transformation is \emph{almost invariant}. Indeed, for $0<a<1/2$ consider $c=1/2-a$. With this choice we always have $b=1/2$. The values of the continued fractions \eqref{FFF1} and \eqref{FFF2} depend on the value of $a$. We have two situations:
\begin{itemize}
\item If $0<a\leq1/4$, then $H=H'=1-2a$. Therefore the only choice of the parameter $y_0$ in order to have a stochastic factorization is $y_0=1-2a$. For this value we always have
\begin{equation}\label{coeff2}
\begin{split}
y_n&=1-2a, \quad n\geq0,\quad y_{-n}=2a, \quad n\geq1,\\
s_n&=r_n=1/2,\quad x_n=1-y_n,\quad n\in\ZZ.
\end{split}
\end{equation}
Therefore the transition probabilities of the Darboux transformation are exactly the same as the original case except for the state 0, where we have
\begin{equation}\label{newc}
\tilde c_0=a,\quad \tilde a_0=a,\quad \tilde b_0=1-2a.
\end{equation}
The random walk generated by $\widetilde P$ is almost the same as the original one except for the state 0. The spectral matrix is given in this case by
\begin{equation}\label{specs}
\widetilde\Psi(x)=\frac{\sqrt{(x-\sigma_-)(\sigma_+-x)}}{2\pi cx(1-x)}\left[\widetilde B+\widetilde Cx\right]+\widetilde{\bm M}_1\delta_1(x),\quad \sigma_{\pm}=1/2\pm\sqrt{2a(1-2a)},
\end{equation}
where
\begin{align*}
\widetilde B&=(1-2a)\begin{pmatrix}1&-\D\frac{1-2a}{2a} \\ -\D\frac{1-2a}{2a}&\D\frac{(1-2a)^2}{4a^2}\end{pmatrix},\quad \widetilde C=\frac{1-2a}{2a}\begin{pmatrix}0&1\\1&-\D\frac{1-4a}{2a}
\end{pmatrix},\quad \widetilde{\bm M}_1=\frac{1-4a}{2(1-2a)}\begin{pmatrix}1&1\\1&1\end{pmatrix}.
\end{align*}
Similar results hold for the LU case.
\item If $1/4<a<1/2$, then $H=1/2$ and $H'=(1-2a)/4a$. Therefore the parameter $y_0$ can be chosen in the range 
$$
(1-2a)/4a\leq y_0\leq 1/2.
$$
If we take $y_0=1-2a$ (which satisfies the previous bounds) then we are in the same situation of the previous case, i.e. we have \eqref{coeff2} for the sequences $x_n,y_n,s_n,r_n$ and the transition probabilities of the Darboux transformation are exactly the same as the original case except for the state 0, where we have again \eqref{newc}. The spectral matrix is then given by \eqref{specs}. If $y_0=1/2$ then we get invariance on the positive states of the Darboux random walk but not on the negative states nor state 0. On the contrary, if $y_0=(1-2a)/4a$, then we get invariance on the negative states of the Darboux random walk but not on the nonnegative states. The spectral matrix can be computed from \eqref{spmat12}. Similar results hold for the LU case.
\end{itemize}

\end{document}